\documentclass[oneside]{amsart}
\usepackage{amsmath,amssymb,amsthm}
\usepackage{geometry}
\usepackage[all]{xypic}

\newcommand{\comment}[1]{}
\newcommand{\clc}{,\ldots,}
\newcommand{\smallcaps}[1]{\textrm{\textsc{#1}}}

\newcommand{\brac}[1]{\left( #1 \right)}
\newcommand{\bbrac}[1]{\bigl( #1 \bigr)}
\newcommand{\cbrac}[1]{\left\{ #1 \right\}}
\newcommand{\sbrac}[1]{\left[ #1 \right]}

\newcommand{\mbrac}[1]{\left| #1 \right|}

\newcommand{\sym}{\mathbb{S}}
\newcommand{\sgn}{\text{sgn}}

\newcommand{\As}{\text{As}}
\newcommand{\Lie}{\text{Lie}}
\newcommand{\PLie}{\text{PLie}}
\newcommand{\PostLie}{\text{PostLie}}
\newcommand{\ComTrias}{\text{ComTrias}}
\newcommand{\Dend}{\text{Dend}}
\newcommand{\Zinb}{\text{Zinb}}

\newcommand{\Nil}{\text{Nil}}
\newcommand{\Mag}{\text{Mag}}
\newcommand{\Com}{\text{Com}}
\newcommand{\Perm}{\text{Perm}}
\newcommand{\coAs}{\text{coAs}}
\newcommand{\coLie}{\text{coLie}}
\newcommand{\coLeib}{\text{coLeib}}

\newcommand{\coPLie}{\text{coPLie}}
\newcommand{\coCom}{\text{coCom}}
\newcommand{\coPerm}{\text{coPerm}}
\newcommand{\coDias}{\text{coDias}}
\newcommand{\susp}{\Sigma}
\newcommand{\quis}{\simeq}
\newcommand{\antis}{{\text{!`}}}
\newcommand{\cc}{\circ}
\newcommand{\cco}{\circ_1\!}
\newcommand{\co}{\bullet}
\newcommand{\ca}{\ast}
\newcommand{\lieg}{\mathfrak{g}}
\newcommand{\operadscript}[1]{\mathcal{#1}}
\newcommand{\CO}{{\operadscript{O}}}
\newcommand{\CP}{{\operadscript{P}}}
\newcommand{\CQ}{{\operadscript{Q}}}
\newcommand{\CM}{{\operadscript{M}}}
\newcommand{\CN}{{\operadscript{N}}}
\newcommand{\CA}{{\operadscript{A}}}
\newcommand{\CB}{{\operadscript{B}}}
\newcommand{\CC}{{\operadscript{C}}}
\newcommand{\CD}{{\operadscript{D}}}

\newcommand{\CF}{{\operadscript{F}}}
\newcommand{\CI}{{\operadscript{I}}}
\newcommand{\CJ}{{\operadscript{J}}}
\newcommand{\CR}{{\operadscript{R}}}
\newcommand{\CS}{{\operadscript{S}}}
\newcommand{\CT}{{\operadscript{T}}}
\newcommand{\CV}{{\operadscript{V}}}
\newcommand{\CW}{{\operadscript{W}}}
\newcommand{\CX}{{\operadscript{X}}}
\newcommand{\CY}{{\operadscript{Y}}}

\newcommand{\barCO}{\overline{\CO}}
\newcommand{\LL}{\mathbb{L}}
\newcommand{\LLO}{\LL_\CO}
\newcommand{\LLP}{\LL_\CP}

\newcommand{\UU}{U}
\newcommand{\Uf}{\UU_f}
\newcommand{\UO}{\UU_\CO}
\newcommand{\UP}{\UU_\CP}

\newcommand{\finbij}{\text{Fin}_{\text{bij}}}
\newcommand{\SM}{\text{$\sym$-\smallcaps{mod}}}
\newcommand{\Cyc}{\text{Cyc}}

\DeclareMathOperator{\Hom}{Hom}

\DeclareMathOperator{\id}{Id}
\DeclareMathOperator{\gr}{gr}

\theoremstyle{plain}

\newtheorem{theorem}{Theorem}[section]
\newtheorem{lemma}[theorem]{Lemma}
\newtheorem{proposition}[theorem]{Proposition}
\newtheorem{corollary}[theorem]{Corollary}

\theoremstyle{remark}
\newtheorem{remark}[theorem]{Remark}
\newtheorem{definition}[theorem]{Definition}
\newtheorem{example}[theorem]{Example}

\title{Operadic comodules and (co)homology theories}
\author{James T. Griffin}

\begin{document}
\begin{abstract}
An operad describes a category of algebras and a (co)homology theory for these algebras may be formulated using the homological algebra of operads.
A morphism of operads $f:\CO\rightarrow\CP$ describes a functor allowing a $\CP$-algebra to be viewed as an $\CO$-algebra.
We show that the $\CO$-algebra (co)homology of a $\CP$-algebra may be represented by a certain operadic comodule.
Thus filtrations of this comodule result in spectral sequences computing the (co)homology.

As a demonstration we study operads with a filtered distributive law; for the associative operad we obtain a new proof of the Hodge decomposition of the Hochschild cohomology of a commutative algebra.
This generalises to many other operads and as an illustration we compute the post-Lie cohomology of a Lie algebra.
\end{abstract}

\maketitle
\section{Introduction}

Hochschild homology and cohomology is defined for associative algebras; when the algebra,~$A$ is commutative and defined over a characteristic~0 field, the (co)homology splits into pieces,
\[
H\!H_n(A) \cong H\!H_n^{(1)}(A) \oplus \ldots \oplus H\!H_n^{(n)}(A).
\]
The first term is isomorphic to the Harrison homology and this splitting was obtained by Barr~\cite{Barr1968}, the full decomposition follows from results of Quillen~\cite{Quillen1970} and has also been described by Gerstenhaber and Schack~\cite{Gerstenhaber1987}.
It has been extended to cyclic (co)homology~\cite{Loday1989},~\cite{Natsume1989} and to $C_\infty$-algebras~\cite{Hamilton2004}.
In this paper we explain the Hodge decomposition using the language of operads.
This is accomplished by a general theory, in which (co)homology theories for algebras defined over a Koszul operad are studied using comodules defined over the Koszul dual cooperad.
The foundational results are Lemmas~\ref{lemma_barlemma} and~\ref{lemma_LLM}.

An operad $\CP$ describes a category of $\CP$-algebras and a morphism of operads $f:\CO\rightarrow\CP$ describes how a $\CP$-algebra $A$ restricts to an $\CO$-algebra, denoted $f^\ast A$.
For instance there is a morphism $\pi:\As\rightarrow\Com$ between the operads describing associative and commutative associative algebras.
The morphism $\pi$ merely observes that an algebra with a commutative associative product is also an associative algebra with that same product.
Another familiar example $i:\Lie\rightarrow\As$ describes how an associative algebra $A$ is naturally a Lie algebra $i^\ast A$ with the bracket $[a,b] = ab - ba$ for $a,b\in A$.
The problem we address is:
\begin{quote}
If $f:\CO\rightarrow\CP$ is a morphism of operads and $A$ is a $\CP$-algebra, what can be said about the $\CO$-algebra homology and cohomology of~$f^\ast A$?
\end{quote}
The homology and cohomology theories of $\CO$-algebras were introduced in~\cite{Ginzburg1994}.  
For an excellent treatment see~\cite{Loday2012} which also gives an introduction to algebraic operads.
Our most general answer to this problem is provided by Theorem~\ref{thm_cohomologyhodge}.

We restrict our attention to operads which are Koszul; a property held by the majority of operads of interest.
In this case the homology and cohomology theories for $\CO$ can be described in terms of the Koszul dual cooperad $\CO^\antis$, see~\cite{Ginzburg1994}.
After introducing the required theory, we will show that for each $\CO$-algebra~$A$ there is a functor from the category of comodules over $\CO^\antis$ to the category of chain complexes; under this functor the (co)homology of~$A$ is represented by~$\CO^\antis$ viewed as a comodule over itself.
In the event that $A\cong f^\ast B$ for a $\CP$-algebra $B$ and for an operad map $f:\CO\rightarrow\CP$, the $\CO$-algebra (co)homology of~$A$ is represented by $\CO^\antis$ viewed as a comodule over~$\CP^\antis$.
Thus a decomposition of (co)homology groups may be obtained from a decomposition of $\CP^\antis$-comodules.  
This is precisely the situation for the map $\pi:\As\rightarrow\Com$ where the Koszul dual $\As^\antis\cong\Sigma\coAs$ decomposes as a comodule over~$\Com^\antis\cong\Sigma\coLie$, inducing the Hodge decomposition of Hochschild (co)homology for a commutative algebra.
Often instead of a decomposition we find a filtration of comodules resulting in a spectral sequence for the (co)homology.

To apply these results we require knowledge of the comodule theories associated to cooperads, or equivalently the module theories associated to operads.
First, there is a subtlety in the definition of left module that must be addressed;
there are two variations, the linear version that we require and a non-linear version which is equivalent to the definition of an algebra in the category of $\sym$-modules.
Most previous results on left modules have used this second definition, for instance the result that a free pre-Lie algebra is free as a Lie-algebra~\cite{Chapoton2010} is proved by showing that the pre-Lie operad itself is a free Lie algebra.
There is no ambiguity for right modules and their properties and associated functors have been studied by Fresse~\cite{Fresse2009}.
There are a number of tools available for the study of modules over operads:
\begin{itemize}
\item One may work directly with generators and relations, an approach that has been furthest developed by Dotsenko and Khoroshkin~\cite{Dotsenko2010grobner}, see also~\cite{Dotsenko2011freeness}.
\item One may combinatorially construct endomorphisms of modules which may then be used to decompose the module via the endomorphism's eigenvalues, we use this approach in Section~\ref{section_equivalence}.
\item One may use the Hadamard product with an operad $\CP$ to transfer structure from $\CO$-modules to $\CO\otimes\CP$-modules.  This works well when $\CP$ is the operad for permutative algebras, see~\cite{Chapoton2001endo}, and we use this technique in Section~\ref{section_dendzinb}.
\item Finally in the case of an operad described using a filtered distributive law~\cite{Dotsenko2011fdls}, one may obtain a strong description of the module theory.
This is the main method used in this article and we develop the theory in Section~\ref{section_fdls}.
\end{itemize}

An operad $\CO$ with a filtered distributive law has two operad maps $i:\CQ \hookrightarrow \CO$ and $f:\CO\rightarrow\CP$, see~\cite{Dotsenko2011fdls}.
We describe two operad filtrations on $\CO$, one positive and one negative, with isomorphic associated graded operads (after removal of the gradings).
However the positive filtration yields a filtration of $\CO$ viewed as a $\CQ$-bimodule, while the negative filtration yields a filtration of $\CO$ viewed as a bimodule over itself.
The Koszul dual of $\CO$ also satisfies a filtered distributive law and thus there are two filtrations of $\CO^\antis$ as a comodule, over~$\CP^\antis$ and over~$\CO^\antis$ itself.
The first gives a spectral sequence for the $\CO$-algebra (co)homology of $f^\ast A$ for a $\CP$-algebra~$A$.
The second gives a spectral sequence for the $\CO$-algebra (co)homology of any $\CO$-algebra~$B$ where the first page is given solely in terms of the $\CQ$-algebra structure $i^\ast B$.
We give conditions when the first spectral sequence collapses; this includes the case of $\pi:\As\rightarrow\Com$ in characteristic $0$ and when $A$ is a free $\CP$-algebra.
In this second case, if \(A = \CP\circ V\) is a free \(\CP\)-algebra and \(M\) some \(A\)-module, then the cohomology of \(A\) with coefficients in \(M\) is isomorphic to \(\Hom_k(\CQ^\antis\circ V, M)\), see Theorem~\ref{thmFree} for further details.

As an example we take the operad morphism $f:\PostLie\rightarrow\Lie$. 
A post-Lie algebra formalises the algebra of vector fields with a flat connection with constant torsion~\cite{MuntheKaas2013}.
We show that the (co)homology of a Lie algebra $\mathfrak{g}$ viewed as a post-Lie algebra $f^\ast\mathfrak{g}$ with zero connection splits into two pieces and we calculate these pieces in terms of the Lie algebra (co)homology of $\mathfrak{g}$.

We finish by treating the morphism $\pi:\As\rightarrow\Com$, proving that our methods recover the classical Hodge decomposition of the Hochschild (co)complex.
We also treat the morphism $\pi':\Dend\rightarrow\Zinb$ recovering a decomposition of the dendriform (co)homology of a zinbiel algebra as observed in~\cite{Goichot}.

This paper may be divided into two parts of approximately equal length.
The first is concerned with (co)homology theories for algebras over Koszul operads and with how they relate to categories of operadic comodules.
It consists of three sections: in Section~\ref{section_background} we give some background theory on operads;
in Section~\ref{section_barhomology} we describe our theory in the easier case of the bar homology of an algebra and cover cyclic homology;
in Section~\ref{section_cohomology} we move onto (co)homology in terms of the cotangent complex and prove the main result, Theorem~\ref{thm_cohomologyhodge}.

The second part supplies examples of decompositions and filtrations of operadic modules.
The main results are in Section~\ref{section_fdls} which develops the module theory associated to filtered distributive laws;
finishing with Section~\ref{section_postlie} treating the specific example of a Lie algebra viewed as a post-Lie algebra;
then the final Section~\ref{section_equivalence} proves that the standard definition of the Hodge decomposition of the Hochschild (co)homology of a commutative algebra agrees with the definition using a filtered distributive law.

\section{Background on operads}
\label{section_background}
In this section we will review some existing constructions and their most important properties.  
The background on cohomology theories (and operads in general) can be found in the self-contained book~\cite{Loday2012}, an original reference is~\cite{Ginzburg1994}.
For some background on the theory of right modules over operads see~\cite{Fresse2009}.
Throughout we will work with differentially graded modules over a field $k$, these form a symmetric monoidal category, $(\mathcal{E}, \otimes, k)$.
The differential on the tensor product of two dg-modules is given using the usual Koszul sign convention.
The suspension of a dg-module $M$ is denoted $sM$, so $(sM)_i = M_{i-1}$.

An $\sym$-module is a functor from the category $\finbij$ of finite sets and bijections to the category of dg-modules $\mathcal{E}$.
The category of $\sym$-modules (with natural transformations as morphisms) is denoted $\SM$; this is a monoidal category with the \emph{composition product} $\circ$.  
The composition product of two $\sym$-modules $\CA$ and $\CB$ is defined by
\begin{equation}\label{eq_compositionproduct}
\CA \circ \CB (X) = \bigoplus_{r\geq 1} \CA(\cbrac{1\clc r}) \otimes_{\sym_r}\bigoplus_{X = X_1\amalg\ldots\amalg X_r} \bigotimes_{i=1}^r \CB(X_i),
\end{equation}
where the second direct sum is indexed by partitions of~$X$ into~$r$ subsets, some of which could be empty.
The unit $\CI$ is the $\sym$-module that is $0$ except on singleton sets which are taken to the unit $k$ of $\mathcal{E}$.
The category $\SM$ is abelian, however the composition product is not linear in the right term.  To correct for this we define the \emph{infinitesimal composition}.
Let $\CA$, $\CB_1$ and $\CB_2$ be $\sym$-modules, then the infinitesimal composite $\CA\circ (\CB_1;\CB_2)$ is the sub-$\sym$-module of $\CA\circ(\CB_1\oplus\CB_2)$ which is linear in the $\CB_2$ term.
Explicitly this is given by
\[
\CA \circ (\CB_1;\CB_2) (X) = \bigoplus_{r\geq 1} \CA(\cbrac{1\clc r}) \otimes_{\sym_r}\!\! 
\bigoplus_{X = X_1\amalg\ldots\amalg X_r} 
\bigoplus_{j=1\clc r} \CB_2(X_j) \otimes \bigotimes_{i=1, i\neq j}^r \CB_1(X_i).
\]
We define $\CA\cco \CB$ to be $\CA\circ(\CI,\CB)$ and call this the \emph{bilinear composition product}.  Explicitly this is
\[
(\CA\cco\CB)(X) = \bigoplus_{Y\subseteq X} \CA(X/Y)\otimes \CB(Y),
\]
where $X/Y = X\setminus Y\amalg\!\cbrac{p}$, the subset $Y$ is replaced by a single point.
It is important to note that the bilinear composition product is not monoidal since it is not associative.
There is however a natural transformation
\begin{equation}\label{eq_circ1assoc}
\xymatrix{\CA\cco ( \CB\cco \CC) \ar[r] & (\CA\cco\CB)\cco\CC,}
\end{equation}
defined for each triple of $\sym$-modules $\CA$, $\CB$, $\CC$.
Each of these is an injection and the domain splits with cokernel
\begin{equation}\label{eq_circ2}
\CA\circ_2(\CB,\CC) := \bigoplus_{Y_1,Y_2\subseteq X\text{ with }Y_1\cap Y_2=\emptyset}\CA(X/Y_1/Y_2) \otimes \CB(Y_1) \otimes \CC(Y_2),
\end{equation}
where $X/Y_1/Y_2$ is the set given by replacing the subsets $Y_1,Y_2$ of $X$ with a point each.
There are symmetries
\[
\alpha_{\CA,\CB,\CC}:\CA\circ_2(\CB,\CC) \rightarrow \CA\circ_2(\CC,\CB),
\]
for \(\sym\)-modules \(\CA,\CB\) and \(\CC\), and there are isomorphisms
\[
(\CA \circ (\CB_1;\CB_2)) \circ \CC \, \cong \, \CA\circ (\CB_1\circ\CC;\CB_2\circ\CC),
\]
for $\sym$-modules $\CA$, $\CB_2$, $\CB_2$ and $\CC$.  In particular
\[
(\CA \cco \CB)\circ \CC \, \cong \, \CA\circ (\CC;\CB\circ\CC).
\]
Now we treat infinitesimal composites of morphisms.
There are natural maps 
\[
\xymatrix{ \CA \ar@<0.5ex>[r]^(.4){\alpha_\CA} & \CA\circ(\CI;\CI) \ar@<0.5ex>[l]^(.6){\beta_\CA} }
\]
allowing us to relate the composition product with its infinitesimal version.  Note that $\CA\circ(\CI;\CI)(X)$ is isomorphic to
\[
\bigoplus_{x\in X} \CA(X).
\]
The map $\beta_\CA$ evaluates the sum of the terms, whereas the map $\alpha_\CA$ takes an element $a$ to $\sum_{x\in X} a$.
Applying $(-\circ\CB)$ to this diagram we get
\[
\xymatrix{ \CA\circ\CB \ar@<0.5ex>[r]^(.4){\alpha_\CA \circ \id} & \CA\circ(\CB;\CB) \ar@<0.5ex>[l]^(.6){\beta_\CA\circ\id} }
\]
This is used to define the infinitesimal composite of morphisms; let $f:\CA_1\rightarrow\CA_2$ and $g:\CB_1\rightarrow \CB_2$ then $f \circ' g$ is defined as
\[\xymatrix{
\CA_1 \circ \CB_1 \ar[r]^(.4){\alpha_{\CA_1}\circ\id} & 
\CA_1\circ(\CB_1;\CB_1) \ar[r]^{f\circ(\id;g)} &
\CA_2 \circ(\CB_1;\CB_2). }
\]
Now suppose that $g:\CB\rightarrow\CB$ is an endomorphism, then we abuse notation by defining $f\circ' g$ to be 
\[\xymatrix{
\CA_1 \circ \CB \ar[r]^(.4){\alpha_{\CA_1}\circ\id} & \CA_1\circ(\CB;\CB) \ar[r]^{f\circ(\id;g)} &
\CA_2 \circ(\CB;\CB) \ar[r]^(.6){\beta_{\CA_2}\circ\id} & \CA_2\circ\CB.}
\]
So if $g_1,g_2:\CB\rightarrow\CB$ are two endomorphisms then we have the linear relation $f\circ'(g_1+g_2) = f\circ' g_1 + f\circ'g_2$, which certainly does not hold for the standard composition product.
An example of this infinitesimal composite comes from the differentials associated to compositions of $\sym$-modules.
If $\CA=(A,d_A)$ and $\CB=(B,d_B)$ are $\sym$-modules with underlying graded $\sym$-modules $A$ and $B$ then $\CA\circ\CB$ is the $\sym$-module with underlying graded $\sym$-module $A\circ B$ and differential 
\[
d_{A\circ B} = d_A\circ\id_B + \id_A\circ' d_B.
\]
This agrees with the definition~\eqref{eq_compositionproduct} where the differential arises from the tensor product of dg-modules in~$\mathcal{E}$.

One way of suspending an $\sym$-module $\CM$ is to suspend each space, so $(s\CM)(n) = s(\CM(n))$.
There is another suspension functor $\susp$ defined on $\CM$ by $\susp\CM(n) = s^{n-1}\CM(n)$.  An equivalent definition is
\[
\susp\CM = (s^{\!-1\!}\CI) \circ \CM \circ (s\CI).
\]
With this form it is clear that $\susp(\CV\circ\CW) \cong \susp\CV\circ\susp\CW$.

\subsection{Operads and their modules}
As discussed above we have a monoidal category $(\SM, \circ,\CI)$.
Algebras in this monoidal category are $\sym$-modules $\CO$ equipped with maps 
\[\xymatrix{
\CO\circ\CO \ar[r]^(.6){\mu_\CO} & \CO & \text{and} & \CI \ar[r]^{u_\CO} & \CO, }
\]
satisfying the associativity square
\begin{equation}\label{eq_operadassociativity}\xymatrix{
\CO\circ\CO\circ \CO \ar[r]^(.6){\id\circ\mu_\CO} \ar[d]_{\mu_\CO\circ\id} & \CO\circ\CO\ar[d]^{\mu_\CO} \\
\CO\circ\CO \ar[r]_(.55){\mu_\CO} & \CO }
\end{equation}
and the unit triangles
\begin{equation}\label{eq_operadunit}\xymatrix{
\CI\circ\CO\ar[r]^{u_\CO\circ\id} \ar[dr] & \CO\circ\CO\ar[d]^{\mu_\CO} & \CO\circ\CI\ar[l]_{\id\circ u_\CO} \ar[dl] \\
& \CO. & }
\end{equation}
An algebra for the composition product in the category of $\sym$-modules is an~\emph{operad}.

For any pair of $\sym$-modules $\CA$, $\CB$ where $\CB$ is equipped with a map $u:\CI\rightarrow \CB$ there is a natural map from the infinitesimal composition product to the full composition product,
\[\xymatrix{
\CA\cco\CB \ar@{}[r]|(.45)*+{\cong} & \CA\circ(\CI;\CB) \ar[r]^{\id\circ(u;\id)} & \CA\circ(\CB;\CB)\ar[r]^(.55){\beta_\CA\circ\id} & \CA\circ \CB. }
\]
Denote by $\mu^1_\CO$ the restriction of the operad composition map $\mu_\CO$ to the bilinear composition product:
\[\xymatrix{
\CO\cco\CO\ar[r] & \CO\circ \CO \ar[r]^(.55){\mu_\CO} & \CO.}
\]
This still determines the operad structure because although the bilinear composition product is not associative, successive powers $(-\cco\CO)^n (\CO)$ cover $\CO\circ\CO$.
Indeed examples of operads are often defined in terms of these bilinear maps.

We will now address the subject of modules for operads.  
There are the concepts of left and right modules for an operad coming from the monoidal category $(\SM,\circ,\CI)$.  For right modules this is precisely the approach we take.
A \emph{right module for an operad~$\CO$} is an $\sym$-module $\CM$ equipped with a composition map $\mu^R_\CM:\CM\circ\CO \rightarrow \CM$.
This must satisfy a unit triangle and the associativity square
\[\xymatrix{
\CM\circ\CO\circ\CO\ar[r]^(.6){\mu^R_\CM\circ\id} \ar[d]_{\id\circ\mu_\CO} & \CM\circ\CO \ar[d]^{\mu^R_\CM} \\
\CM\circ\CO \ar[r]_(0.55){\mu^R_\CM} & \CM. }
\]
In the same way as for an operad the structure may be defined using a bilinear composition product $\CM\cco\CO\rightarrow \CM$.

However for left modules there is a choice.  We will define a \emph{left module for an operad~$\CO$} to be an $\sym$-module $\CM$ equipped with a map $\mu^L_\CM:\CO\cco\CM\rightarrow \CM$ from the {\bf bilinear} composition product.
This must satisfy the pentagon
\[\xymatrix{
&\CO\cco (\CO\cco \CM) \ar[r]^(.6){\id\circ\mu^L_\CM}\ar[dl] & \CO\cco\CM \ar[dd]^{\mu^L_\CM} \\
(\CO\cco\CO)\cco\CM \ar[d]^{\mu_\CO\cco\id} && \\
\CO\cco\CM \ar[rr]^{\mu^L_\CM} && \CM.
}\]
The top-left diagonal map is a natural inclusion~\eqref{eq_circ1assoc}.
This definition using the bilinear composition product differs substantially from the definition using the regular composition product.
To distinguish between the two versions we make the following definition: if $\CA$ is an $\sym$-module equipped with a composition map $\mu_\CA:\CO\circ\CA\rightarrow\CA$ satisfying a square
\[\xymatrix{
\CO\circ\CO\circ \CA \ar[r]^(.6){\id\circ\mu_\CA} \ar[d]_{\mu_\CO\circ\id} & \CO\circ\CA\ar[d]^{\mu_\CA} \\
\CO\circ\CA \ar[r]_(.55){\mu_\CA} & \CA, }
\]
then we say that $\CA$ is an \emph{$\sym$-module-algebra} for $\CO$.
Most of our examples of $\sym$-module-algebras for $\CO$ will concern $\sym$-modules which are ``concentrated in arity 0'', meaning that they are $\sym$-modules which evaluate as 0 except on the empty set, where they give a dg-module; we refer to these simply as \emph{$\CO$-algebras}.

Finally we address the category of bimodules for an operad $\CO$.
Let $\CM$ be an $\sym$-module equipped with maps $\mu_\CM^R:\CM\cco\CO\rightarrow\CM$ and $\mu_\CM^L:\CO\cco\CM\rightarrow\CM$, which make $\CM$ into a right and a left $\CO$-module respectively.
Then $\CM$ is an \emph{$\CO$-bimodule} if it also satisfies the compatibility pentagon
\[\xymatrix{
&\CO\cco (\CM\cco \CO) \ar[r]^(.6){\id\circ\mu^R_\CM}\ar[dl] & \CO\cco\CM \ar[dd]^{\mu^L_\CM} \\
(\CO\cco\CM)\cco\CO \ar[d]^{\mu^L_\CM\cco\id} && \\
\CM\cco\CO \ar[rr]^{\mu^R_\CM} && \CM,
}\]
and another compatibility diagram
\[\xymatrix{
\CO\circ_2(\CO,\CM) \ar[d]\ar[r] & (\CO\cco\CO)\cco\CM \ar[r]^(.6){\mu_\CO\cco\id} & \CO \cco \CM\ar[r]^(.58){\mu_\CM^L} & \CM\\
\CO\circ_2(\CM,\CO) \ar[r] & (\CO\cco\CM)\cco\CO \ar[r]^(.6){\mu_\CM^L\cco\id} & \CM \cco \CO, \ar[ur]_{\mu_\CM^R} &
}\]
where the unlabelled arrows are natural transformations arising from the category of $\sym$-modules.
Since the product $\cco$ is bilinear all three categories; left modules, right modules and bimodules over~$\CO$ are abelian.  
Hence these categories behave much like the category of modules over some ring.
This is in contrast with the category of $\sym$-module-algebras for $\CO$ which is not abelian.

\subsection{Cooperads, comodules and coalgebras}
A \emph{cooperad} is an $\sym$-module $\CC$ equipped with a comultiplication $\Delta_\CC:\CC\rightarrow \CC\circ\CC$ and a counit $\varepsilon_\CC:\CC\rightarrow \CI$.
It must satisfy a coassociativity square and counit triangles which are given by considering the diagrams~\eqref{eq_operadassociativity} and~\eqref{eq_operadunit} and reversing the arrows.
Similarly for left, right and bi~comodules and for $\sym$-module-coalgebras for a cooperad; these are defined by taking the dual definitions of left, right and bi~modules and of $\sym$-module-algebras respectively.
To avoid the term `bicomodule' (or indeed `cobimodule') we will use the term \emph{comodule} for the dual notion of bimodule.
This should not cause confusion because for operads one should always specify the direction when using left or right comodules.

The linear dual of a cooperad is always an operad and similarly for the notions of modules and algebras.
When an operad $\CO$ is finite dimensional in the sense that $\CO(X)$ is finite dimensional for each set $X$ then the linear dual is a cooperad.
Again there are similar statements for modules and algebras.

\subsection{Twisting morphisms}
The homological algebra of operads was studied in~\cite{Ginzburg1994} and~\cite{Getzler1994operads}.  Here we recap the required basic theory.
For a full account see the book~\cite{Loday2012}.

An \emph{operadic twisting morphism} between a cooperad $\CC$ and an operad $\CO$ is a map of $\sym$-modules
\[
\kappa: \CC\rightarrow\CO
\]
of differential degree -1.  It must satisfy the equation
\[
d(\kappa) + \kappa\ast\kappa = 0,
\]
where $\kappa\ast\kappa$ is defined to be
\[\xymatrix{
\CC \ar[r]^(.4){\Delta_\CC} & \CC\cco\CC \ar[r]^{\kappa\cco\kappa} & \CO\cco\CO \ar[r]^(.6){\mu_\CO} & \CO.}
\]
This definition is equivalent to the requirement that the map $d_\kappa$ defined as the composite
\begin{equation}\label{eq_ltcp}\xymatrix@C=40pt{
\CC\circ\CO \ar[r]^(0.4){\Delta_\CC\circ\id} & (\CC\cco\CC)\circ\CO \ar[r]^{(\id\cco\kappa)\circ\id} & (\CC\cco\CO)\circ\CO \ar[dll]^\cong \\
 \CC\circ(\CO;\CO\circ\CO)\ar[r]_(.55){\id\circ(\id,\mu_\CO)} & \CC\circ(\CO;\CO) \ar[r]_(0.55){\beta_\CC\circ\id} & \CC\circ\CO }
\end{equation}
gives a differential $d_\kappa + d_{\CC\circ\CO}$.  
The associated chain complex is denoted $\CC\circ_\kappa \CO$ and is called the \emph{left twisted composite product}.

All of the examples of operads in this article will be \emph{quadratic}.
This means that they have a presentation by a generating $\sym$-module $\CV$ along with relations $\CR \subseteq \CV\cco\CV \subseteq\CF\CV$, where $\CF\CV$ denotes the free operad on $\CV$.
Note that the operad $\CF\CV$ is naturally graded by weight, that is the number of copies of $\CV$.  
Then $\CV\cco\CV$ is isomorphic to the subspace of weight 2 elements of $\CF\CV$.
For any quadratic operad $\CO = \CF\CV/(\CR)$ there is a notion of \emph{Koszul dual cooperad}, $\CO^{\antis}$.
This is the cooperad with generating $\sym$-module $\Sigma\CV$ and corelations $\Sigma^2\CR\subseteq \Sigma\CV\cco \Sigma\CV$.
The \emph{Koszul dual operad}, $\CO^!$ is defined to be the suspension of the linear dual $\susp(\CO^\antis)^\ast$.

The most common examples of twisting morphisms come from Koszul duality.
For any quadratic operad $\CO$ there is a twisting morphism, $\kappa$ from $\CO^\antis$ to $\CO$ sending $s\CV$ to $\CV$ and all other elements to $0$.
The operad $\CO$ is said to be \emph{Koszul} if the left twisted composite product $\CO^\antis\circ_\kappa\CO$ is quasi-isomorphic to the unit $\sym$-module $\CI$.
We refer the reader to Chapter~7 of~\cite{Loday2012} for the details and to Chapter~13 for many examples.

\section{Decompositions of bar homologies}
\label{section_barhomology}
Let $f:\CO\rightarrow\CP$ be a morphism of Koszul operads and let $A$ be a $\CP$-algebra.
As a first result the theory of right $\CP^\antis$-comodules is used to give decompositions of the bar homology of the $\CO$-algebra $f^\ast A$.

To start we review the bar complex of an $\CO$-algebra.
In fact we describe a more general complex, $\CM\circ_\kappa A$ for a right $\CO^\antis$-comodule $\CM$ and an $\sym$-module-algebra $A$ over $\CO$.
When the given comodule is $\CO^\antis$ itself and $A$ is an $\CO$-algebra, this specialises to the bar complex $B_\CO(A)$.
Next we state and prove~Lemma~\ref{lemma_barlemma} which will be used throughout this article.
An immediate application is Theorem~\ref{thm_barhodge} which is a tool for calculation of the bar homology of a $\CP$-algebra $A$ viewed as an $\CO$-algebra~$f^\ast A$.

In the final part of this section, the cyclic homology of an associative algebra is shown to be represented as a right $\As^\antis$-comodule.
This is inspired by work of Loday and Quillen~\cite{Loday1984cyclic}; we also discuss a right $\As^\antis$-comodule which represents the Lie algebra homology $H^{\text{CE}}_\ast(\mathfrak{gl}_\infty(A))$ of the infinite general linear group with entries in a unital associative algebra $A$.

\subsection{The bar homology}
Let $\CO$ be an augmented operad $\CO\rightarrow \CI$ with augmentation ideal $\barCO$ and let $A$ be an $\CO$-algebra.
Then the space of indecomposables of $A$ is the quotient
\[
A / \mu_A(\barCO\circ A).
\]
The indecomposables functor is the left adjoint in a Quillen adjunction and the bar complex is its derived functor; the bar homology is then the homology of the bar complex.
When the operad $\CO$ is Koszul then there is a simple description using the Koszul dual cooperad, we will use this as the definition of the bar complex in Example~\ref{example_barhomology}.

There is a generalisation of left twisted composite products for right comodules.
Let $\kappa:\CC\rightarrow\CO$ be a twisting morphism, let $\CM$ be a right $\CC$-comodule and let $\CA$ be an $\sym$-module-algebra over $\CO$.
Then the left twisted composite of $\CM$ and $\CA$ is defined in a similar manner to the left twisted composite of $\CC$ and $\CO$.
The underlying graded $\sym$-module is $\CM\circ\CA$, but it has in addition a differential defined by
\[\xymatrix@C=40pt{
\CM\circ\CA \ar[r]^(0.4){\Delta^R_\CM \circ\id} & (\CM\cco\CC)\circ\CA \ar[r]^{(\id\cco\kappa)\circ\id} & (\CM\cco\CO)\circ\CA \ar[dll]^\cong \\
 \CM\circ(\CA;\CO\circ\CA)\ar[r]_(.55){\id\circ(\id,\mu_\CA)} & \CM\circ(\CA;\CA) \ar[r]_(0.55){\beta_\CM\circ\id} & \CM\circ\CA }
\]
Denoting this composite map by $d_\kappa$, the \emph{left twisted composite} is $\CM\circ_\kappa \CA = (\CM\circ\CA, d_\kappa + d_{\CM\circ\CA})$.
\begin{lemma}
Let $\CM$ and $\CA$ be as above.  Then the left twisted composite $\CM\circ_\kappa\CA$ is a well defined chain complex.
The functor $(-)\circ_\kappa\CA$ is right exact and when the base field~$k$ is of characteristic~0 it is also left exact.
\end{lemma}
\begin{proof}
The product $\CC\circ\CO$ is both a left $\CC$-comodule and a right $\CO$-module and the differential $d_\kappa$ is compatible with this structure and hence $\CC\circ_\kappa\CO$ also holds these structures.
The product $\CM\circ_\kappa\CA$ is isomorphic to
\[
\CM \circ_\CC (\CC\circ_\kappa \CO) \circ_\CO \CA.
\]
The map $\kappa$ is a twisting morphism so $\CC\circ_\kappa\CO$ is a chain complex and hence so is~$\CM\circ_\kappa\CA$.

As the underlying graded module of $\CM\circ_\kappa\!\CA$ is the composition product $\CM\circ\CA$ the exactness only depends on the category of $\sym$-modules.  
By the definition~\eqref{eq_compositionproduct}, for a finite set $X$ the graded vector space $\CM\circ\CA(X)$ is the direct sum
\[
\bigoplus_{r\geq 1} \CM(\cbrac{1\clc r}) \otimes_{\sym_r}\bigoplus_{X = X_1\amalg\ldots\amalg X_r} \bigotimes_{i=1}^r \CA(X_i).
\]
Each term of this direct sum is a tensor product of the right $k\sym_r$-module $\CM(r)$ with a left $k\sym_r$-module.
This is right exact in the $\CM(r)$ term and hence the whole functor $(-)\circ\CA$ is right exact.

Finally, when the characteristic is~$0$ the algebra $k\sym_r$ is semi-simple for each $r$ and so the tensor products are exact, hence the functor $(-)\circ\CA$ is exact.
\end{proof}
\begin{example}
Let $\CM$ be equal to $\CC$ viewed as a right module over itself, likewise let $\CA=\CO$ be viewed as an $\sym$-module-algebra over itself.  Then $\CM\circ_\kappa\CN \cong \CC\circ_\kappa\CO$ is the standard left twisted composite product.
\end{example}
\begin{example}\label{example_barhomology}
Suppose that $\CO$ is a Koszul operad and let $\CC$ be its Koszul dual cooperad.
Let $A$ be an algebra for $\CO$ concentrated in arity 0.
Then $\CC\circ_\kappa A$ is the \emph{bar complex}, $B_\CO(A)$ of the $\CO$-algebra~$A$.
\end{example}
\begin{example}\label{example_barcomplex}
Let $\As$ be the associative operad and $A$ be an associative algebra.  Then the bar homology is computed by $B_\As(A)=\susp\coAs\circ_\kappa A$, or
\[
B_\As(A) = A \oplus s(A\otimes A) \oplus s^2(A\otimes A \otimes A) \oplus\ldots,
\]
with elements traditionally written, see~\cite{Eilenberg1953}, as $[a_1\mid a_2\mid \ldots \mid a_k]$ and differential
\[
d_\kappa ([a_1\mid a_2\mid \ldots \mid a_k]) = \sum_{i=1}^{k-1} (-1)^{i-1+s_i} [a_1\mid\ldots\mid a_i a_{i+1}\mid\ldots\mid a_k]
\]
for homogeneous elements $a_j$ with $s_i=\sum_{j=1}^{i-1}|a_j|$ the sum of degrees.
\end{example}
\begin{remark}
The lemma above gave a condition on the field $k$ under which the functor $(-)\circ_\kappa\!\CA$ is exact.
If $\CA(\emptyset)=0$ then the functor is exact for any field.
This is because the left $k\sym_r$-modules
\[
\bigoplus_{X = X_1\amalg\ldots\amalg X_r} \bigotimes_{i=1}^r \CA(X_i)
\]
are free.  This fact is of central importance in the theory of shuffle operads and Gr\"obner bases of operads~\cite{Dotsenko2010grobner} which only applies to $\sym$-modules which are~$0$ on the empty set.
\end{remark}

\subsection{The relative setup}
Consider a commuting square
\begin{equation}\label{eq_relativesquare}
\xymatrix{
\CO \ar[r]^f & \CP  \\
\CC \ar[u]^\kappa \ar[r]_g & \CD \ar[u]_(.54){\kappa'},
}
\end{equation}
where $f$ is a morphism of operads, $g$ is a morphism of cooperads and $\kappa,\kappa'$ are twisting morphisms.
The primary example is when $\CO$ and $\CP$ are quadratic operads and the map $f$ takes the generators of $\CO$ to the generators of $\CP$.
Then the map $g$ is naturally defined between the Koszul duals $\CC = \CO^\antis$ and $\CD = \CP^\antis$ and forms a commuting square with the natural twisting maps.

\begin{lemma}\label{lemma_barlemma}
Let $\CM$ be a right $\CC$-comodule and let $\CA$ be an $\sym$-module-algebra over~$\CP$.
Write $g^\ast\CM$ for the $\CD$-comodule based on $\CM$ and write $f^\ast\CA$ for the $\sym$-module-algebra over $\CO$ based on~$\CA$.
Then
\begin{equation}\label{eq_barlemma}
\CM\circ_\kappa f^\ast \CA \,\cong\, g^\ast\CM\circ_{\kappa'}\CA.
\end{equation}
\end{lemma}
\begin{proof}
Both $\sym$-modules are equal to $\CM\circ\CN$ with an added differential, $d_\kappa$ on the left hand side and $d_{\kappa'}$ on the right hand side.
We need to prove that these are equivalent.
Consider the diagram,
\begin{equation}
\xymatrix@C=0.9pc{
\CM\circ\CA \ar[r]\ar@/_/[dr]_(.35){\Delta^R_{g^\ast \CM}\circ\id} &(\CM\cco\CC)\circ\CA \ar[r]\ar[d] & (\CM\cco\CO)\circ\CA \ar[r]^\cong\ar[d] &
\CM\circ(\CA;\CO\circ\CA) \ar@/^/[dr]^(.65){\id\circ(\id;\mu_{f^\ast\!\CA})}\ar[d] &\\
&(\CM\cco\CD)\circ\CA \ar[r] & (\CM\cco\CP)\circ\CA \ar[r]^\cong  
&\CM\circ(\CA;\CP\circ\CA) \ar[r] & \CM\circ(\CA;\CA) \ar[d] \\
&&&& \CM\circ\CA.
}
\end{equation}
The composition along the top edge of the diagram is $d_\kappa$, whilst the composition along the bottom is $d_{\kappa'}$.
To see that the diagram commutes we only need to note that the left square is derived from the square~\eqref{eq_relativesquare}, whilst the right square comes purely from the naturality of the composition products.
The diagonal edges are defined as the composites of the other edges of the triangles they span.
\end{proof}
\begin{theorem}\label{thm_barhodge}
Let $f:\CO\rightarrow\CP$ be a map of Koszul operads with $g:\CC\rightarrow\CD$ the corresponding map of Koszul duals.
Let $A$ be a $\CP$-algebra and $f^\ast A$ be its associated $\CO$-algebra.
\begin{enumerate}
\item[(I)] We have the following expression for the bar complex of $f^\ast A$.
\begin{equation}
B_\CO(f^\ast A) \cong g^\ast\CC \circ_{\kappa'} A,
\end{equation}
where $g^\ast \CC$ is the cooperad $\CC$ viewed as a right $\CD$-comodule.
\item[(II)]
Suppose further that $g^\ast\CC$ is free as a right $\CD$-comodule with generating $\sym$-module $\CB$.
Then
\begin{equation}
B_\CO(f^\ast A) \cong \CB \circ B_\CP(A).
\end{equation}
\end{enumerate}
\end{theorem}
\begin{proof}
For part (I) note that when $A$ is a $\CO$-algebra then the twisted composite $\CC\circ_\kappa A$ is isomorphic to the bar complex of~$A$.
Applying Lemma~\ref{lemma_barlemma} with $\CM = g^\ast\CC$ gives the isomorphism.

For part (II) we use part (I), if $g^\ast\CC$ is isomorphic to $\CB\circ\CD$ as a right $\CD$-module then 
\[
B_\CO(f^\ast A) \cong g^\ast \CC\circ_\kappa A \cong (\CB\circ\CD)\circ_\kappa A \cong \CB\circ(\CD\circ_\kappa A) \cong \CB\circ B_\CP (A),
\]
which is the required isomorphism.
\end{proof}
\begin{example}
The Koszul dual pair of operads $\Perm$ and $\PLie$ were studied by Chapoton and Livernet~\cite{Chapoton2001rooted}.
There is a map $\Perm\rightarrow\Com$ of operads encoding the fact that a commutative algebra is naturally a permutative algebra.
In the Koszul dual picture we have a map $\susp\coPLie\rightarrow\susp\coLie$ of cooperads.
Since $\PLie$ is free as a right $\Lie$-module (proved for instance by Theorem~4 of~\cite{Dotsenko2011freeness}) we may apply the second part of Theorem~\ref{thm_barhodge}.
We find that for a commutative algebra $A$, the bar homology of $A$ viewed as a permutative algebra is isomorphic to $\CB\circ B_\Com(A)$ for some $\sym$-module~$\CB$.
\end{example}
We will see many other examples provided by the theory of filtered distributive laws in Section~\ref{section_fdls}.

\subsection{Cyclic homology}
The cyclic homology of an associative algebra is a functor from the category of algebras to graded vector spaces; we will construct a right $\As^\antis$-comodule, for which the generalised twisted composite product with an algebra~$A$ provides Connes' complex~$C^\lambda_\ast(A)$.
In characteristic~0, this computes the cyclic homology~$HC_\ast(A)$.
For definitions of these terms and the functoriality of cyclic homology, see Section~2.1 of~\cite{Loday1998}.

Let~$\CX$ be the~$\sym$-module with~$\CX(n) = k\sym_n$ and $\sym_n$-action by conjugation.
This splits into a sum of~$\sym_n$-modules indexed by the conjugacy classes of~$\sym_n$.
Define a sub-$\sym$-module $\Cyc$ with $\Cyc(n)$ the span of $n$-cycles and denote the cycle~$(1\, 2\,\ldots\, n)$ by $\sigma_n$, the generates the \(k\sym_n\)-module \(\Cyc(n)\).
The dimension of $\Cyc(n)$ is $(n-1)!$ and the action of $(1\,\ldots\, n)\in\sym_n$ on $\sigma_n$ is trivial, hence $\Cyc(n)$ is isomorphic to~$k_{C_n}\otimes_{C_n} k\sym_n$, the trivial module for the cyclic group~$C_n$ induced up to~$\sym_n$.

The sub-$\sym$-module~$\Cyc$ generates~$\CX$ as a commutative algebra in the category of $\sym$-modules where $\CX(n)\otimes\CX(m) \rightarrow \CX(n+m)$ is the inclusion $k\sym_n\otimes k\sym_m \hookrightarrow k\sym_{n+m}$.
In fact $\CX$ is free, so $\CX\cong \Com\circ\Cyc$.
This should be familiar, it is the standard cycle decomposition of a permutation.

By defining $\sigma_n \circ_i m_2 = \sigma_{n+1}$ for $i=1\ldots n$, we put a right $\As$-module structure on $\Cyc$ and hence also on~$\CX$ by extension using~$\CX\cong \Com\circ \Cyc$.
Let $\CY = \Sigma\CX^\ast$ and $\CC=\Sigma\Cyc^\ast$ be the suspensions of the linear duals of $\CX$ and $\Cyc$.
Then both~$\CY$ and~$\CC$ are right $\As^\antis$-comodules and for an associative algebra~$A$ we will study the complexes~$\CY\circ_\kappa A$ and~$\CC\circ_\kappa A$.

We start by noting that $\CC(n) \cong s^{n-1} \sgn_{C_n} \otimes_{C_n}k\sym_n$, where $\sgn_{C_n}$ is the sign representation for $C_n$.
Thus $\CC(n)\otimes_{\sym_n} A^{\otimes n}$ is isomorphic to $s^{n-1}\sgn_n\otimes_{C_n} A^{\otimes n}$ or equivalently the coinvariants of the action of an operator $t$ which acts by cyclically permuting the copies of~$A^{\otimes n}$ along with a sign~$(-1)^n$.
The differential~$d_\kappa$ coincides with the differential of the complex~$C^\lambda_\ast(A)$ of~\cite{Loday1998}, which in characteristic~0, computes the cyclic homology of $A$.  The proof that the differentials coincide is direct; one needs to carefully examine the two definitions, see~\cite{Loday1984cyclic} and Section~10.2 of~\cite{Loday1998} for similar arguments.

For $\CY\circ_\kappa A$ we use the isomorphism $(\Sigma\Com\circ \CC)\circ_\kappa A \cong \Sigma\Com\circ(\CC\circ_\kappa A)$ to find that it is isomorphic to the desuspension of the free cocommutative coalgebra with cogenerators the suspension of~$C^\lambda_\ast(A)$.
This complex was used in~\cite{Loday1984cyclic} in the calculation of the homology of the infinite general linear group with entries in~$A$.

\begin{remark}
For a cyclic operad $\CO$ the dual cooperad $\CC$ is cyclic; one may define a right $\CC$-comodule by $\CC'(n)=\CC(n-1)$.
Then $\CC'\circ_\kappa A$ computes the cyclic homology of a $\CO$-algebra $A$ with coefficients in itself, this is defined in~\cite{Getzler1995cyclic}.
\end{remark}

\begin{remark}
For a Lie algebra $\mathfrak{g}$ there is a right $\coLie$-comodule structure on the $\sym$-module $\CM_\mathfrak{g}$, where $\CM_\mathfrak{g}(n) = \mathfrak{g}^{\otimes n}$.
The bilinear comultiplication $\Delta^1:\CM_\mathfrak{g} \rightarrow \CM_\mathfrak{g} \circ_1 \coLie(2)$ is given by
\[
\Delta^1(g_1\otimes\ldots\otimes g_n) =
\sum_{1\leq i < j\leq n} ([g_i,g_j]\otimes g_1 \otimes\ldots \widehat{g}_i \ldots \widehat{g}_j\ldots\otimes g_n) \otimes l(\cbrac{i,j}).
\]
For a commutative algebra $A$ the left twisted composite $\Sigma\CM_\mathfrak{g}\circ_\kappa A$ is isomorphic to the Chevalley-Eilenberg complex of the Lie algebra $\mathfrak{g}\otimes A$, which can be seen via a direct comparison.
When $A$ is unital and $\mathfrak{g}$ is reductive, the homology is carried on the $\mathfrak{g}$-invariants of the Chevalley-Eilenberg complex and hence the left twisted composite $\Sigma\CM_\mathfrak{g}^\text{inv} \circ_\kappa A$, with $\CM_\mathfrak{g}^\text{inv}(n):= (\mathfrak{g}^{\otimes n})^{\mathfrak{g}}$ a subcomodule of $\CM_\mathfrak{g}$, calculates the homology of~$\mathfrak{g}\otimes A$.

We do not pursue this idea any further but will connect it to the discussion above:
when $\mathfrak{g} = \mathfrak{gl}_\infty$ is the limit of the general linear Lie algebras we have $\text{lim}_{k\rightarrow\infty}\Sigma\CM_{\mathfrak{gl}_k}^{\text{inv}} \cong \CY$, where $\CY$ is the comodule defined above but viewed as a $\Com^\antis$-comodule rather than an $\As^\antis$-comodule.
\end{remark}

\section{Cohomology of algebras}
\label{section_cohomology}
Theorem~\ref{thm_barhodge} showed how the bar homology of a $\CP$-algebra viewed as an $\CO$-algebra could be decomposed in terms of the right $\CP^\antis$-comodule structure of $\CO^\antis$.
We now extend this result to the homology and cohomology theories for algebras over an operad.

To begin we review the definitions of the homology and cohomology theories.
Of central importance to these definitions is the cotangent complex associated to an algebra.
And the cental result is Lemma~\ref{lemma_barlemma} which connects the cotangent complex with the module theory of operads.
This is used in Theorem~\ref{thm_cohomologyhodge} to show how the homology and cohomology theories of a $\CP$-algebra viewed as a $\CO$-algebra can be decomposed by considering the $\CP^\antis$-comodule structure of $\CO^\antis$.

\subsection{Enveloping algebras}\label{section_envelopingalgebras}
Let $\CO$ be an operad and $A$ an $\CO$-algebra concentrated in arity 0.
Then the \emph{enveloping algebra} $\UO A$ is the coequaliser
\begin{equation}
\label{eq_UOcoeq}\xymatrix@C=3.0pc{
\CO\circ(\CO\circ A;\CI) \ar@<1ex>[r]^(.55){\id\circ(\mu_A;\id)} \ar@<-1ex>[r]_(.55){\mu_\CO'} & \CO\circ(A;\CI) \ar@{.>}[r]^(.57){\pi_A} & \UO A
}\end{equation}
where $\mu_\CO'$ uses the multiplication $\mu_\CO$ to contract multiple copies of~$\CO$.
The coequaliser \(\UO A\) is naturally a unital associative algebra.
An $A$-module is defined to be a left $\UO A$-module, details may be found in Chapter~4 of~\cite{Fresse2009}.
\begin{example}
\begin{itemize}
\item Let $\CO$ be the associative operad and $A$ be an associative algebra, then $\UO A$ is isomorphic to $A^e:= k\oplus A\otimes A^{\text{op}}$.  The left modules for this algebra are the bimodules for $A$.
\item If $\CO$ is the Lie operad and $\mathfrak{g}$ a Lie algebra then $\UO \mathfrak{g}$ is isomorphic to the universal enveloping algebra
\[
\UU_\Lie\mathfrak{g} \cong T \mathfrak{g} / (xy - yx - [x,y]_{\mathfrak{g}}\mid x,y\in\mathfrak{g}).
\]
\item If $\CO$ is the commutative operad and $A$ is a commutative algebra then $\UU_\Com A\cong k \oplus A$.
\end{itemize}
\end{example}
If $f:\CO\rightarrow\CP$ is a morphism of operads and $A$ is a $\CP$-algebra then there is a morphism of associative algebras $\Uf A:\UO f^\ast A \rightarrow \UP A$.
This map is constructed by considering the diagram
\begin{equation}\label{eq_UfA}\xymatrix{
\CO\circ(\CO\circ A;\CI) \ar[r]\ar@/^1pc/[rr]^(.55){\mu_\CO'} & \CO\circ(\CP\circ A;\CI) \ar[d]\ar[r] & \CO\circ(A;\CI) \ar[dr]^\gamma \ar[d]\ar[r]^(.55){\pi_{f^\ast\! A}} & \UO f^\ast\! A \ar@{.>}[d]^{\Uf A} \\
& \CP\circ(\CP\circ A;\CI)\ar[r]& \CP\circ(A;\CI) \ar[r]^(.55){\pi_A} & \UP A.
}\end{equation}
The map $\gamma$ coequalises the diagram~\eqref{eq_UOcoeq} and so uniquely defines $\Uf A$.
The map $\Uf A$ observes that if $M$ is a module for $A$ viewed as an $\CO$-algebra, then $M$ is still a module for $f^\ast A$ as a $\CP$-algebra.
We denote this $f^\ast A$-module by $f^\ast M$.

\subsection{Cotangent complexes}\label{section_cotangentcomplex}
Suppose that $\CO$ is a Koszul operad with Koszul dual cooperad $\CC$ and that $A$ is an $\CO$-algebra.
We will define the cotangent complex $\LLO A$ via an explicit construction.  In~\cite{Loday2012} it is defined by taking the K\"ahler differentials of a cofibrant replacement of the algebra $A$.  
We choose an explicit model for two reasons; (1) it simplifies the exposition of this article, we do not need to recall all the background theory to prove our results and (2) it is through this explicit model of the cotangent complex that we are able to see the desired structure.

We start with the free left $\UO A$-module generated by the bar complex of $A$,
\[
\UO A \otimes B_\CO (A) \cong \UO A \otimes (\CC\circ_\kappa A).
\]
Then we add an extra differential, defined on the generators by
\[\xymatrix{
\CC\circ A \ar[r]^(.4){\Delta_\CC\circ\id} & (\CC\cco\CC)\circ A \ar[r]^(.42){\cong} & \CC\circ(A;\CI)\otimes (\CC\circ A) \ar[dl]_{\kappa\circ\id} \\
&\CO\circ(A;\CI)\otimes(\CC\circ A) \ar[r]^(.53){\pi_A\otimes\id} & \UO A \otimes (\CC\circ A).
}\]
Calling this differential $d^l$ we have the \emph{cotangent complex}
\[
\LLO A := (\UO A\otimes (\CC\circ_\kappa \! A),\, d^l) \cong (\UO A\otimes (\CC\circ A),\, d^l + \id \otimes \, d_\kappa).
\]
We are now ready to define the homology and cohomology of an $\CO$-algebra $A$.
Let $M$ be a left $\UO A$-module and let $N$ be a right $\UO A$-module.
Then the \emph{cohomology of $\UO A$ with coefficients in $M$} is defined to be
\[
H_\CO^\ast(A,M) := H^\ast(\Hom^\ast_{\UO A}(\LLO A, M)),
\]
and the \emph{homology of $\UO A$ with coefficients in $N$} is defined to be
\[
H^\CO_\ast(A,M) := H_\ast(N\otimes_{\UO A} \LLO A).
\]
\begin{remark}
\label{rem_HH}
When $\CO$ is the associative operad these theories are closely related to Hochschild cohomology and homology, see~\cite{Loday1998}.
The precise relation is given as follows: for $i\geq 2$,
\[
H\!H^i(A,M) \cong H^{i-1}_\As(A,M)\quad\text{ and }\quad H\!H_i(A,M) \cong H_{i-1}^\As(A,M).
\]
This does not hold for $i=1$ as the Hochschild (co)complex is truncated.  For example $H_\As^0(A,M)$ is equal to the derivations of $A$ with coefficients in $M$, whereas $H\!H^1(A,M)$ is given by the derivations modulo the inner derivations.
Finally the zeroth Hochschild (co)homology groups are not represented at all.

There is a similar relationship when $\CO = \Com$ with the Harrison (co)homology of commutative algebras and when $\CO=\Lie$ with the (co)homology of Lie algebras.
\end{remark}

\subsection{The generalised cotangent complex associated to a comodule, $\CM$}
Let $\CO$ be a Koszul operad and $\CC$ its Koszul dual cooperad.  Let $\CM$ be a $\CC$-comodule.
Then for an $\CO$-algebra~$A$ we define $\LLO^\CM A$ to be
\[
\LLO^\CM A = (\UO A\otimes (\CM\circ_\kappa\! A),\, d^l) = (\UO A\otimes (\CM\circ A),\, d^l + \id \otimes\, d_\kappa),
\]
where $d^l$ is defined on generators by
\begin{equation}\label{eq_Mcotangent}\xymatrix{
\CM\circ A \ar[r]^(.4){\Delta^L_\CC\circ\id} & (\CC\cco\CM)\circ A \ar[r]^(.44){\cong} & \CC\circ(A;\CI)\otimes (\CM\circ A) \ar[dl]_{\kappa\circ\id} \\
&\CO\circ(A;\CI)\otimes(\CM\circ A) \ar[r]^(.53){\pi_A\otimes\id} & \UO A \otimes (\CM\circ A)
}\end{equation}
and then extended.
\begin{lemma}
\label{lemma_LLM}
This does define a chain complex $\LLO^\CM A$.
The functor $\LLO^{(-)}A$ from the category of $\CC$-comodules to the category of $\UO A$-modules is right exact and when the characteristic of the base field~$k$ is~$0$ the functor is also left exact.
\end{lemma}
\begin{proof}
The algebra $\CO\circ(A;\CI)$ maps to $\UO A$ via $\pi_A$, so it will be sufficient to prove that $\CO\circ(A;\CI)\otimes(\CM\circ_\kappa\!\CA)$ is a chain complex with the differential defined by the first three arrows of~\eqref{eq_Mcotangent}.
Then $\LLO^\CM A$ may be obtained from this by a change of coefficients along~$\pi_A$.
Consider the diagram
\[\xymatrix@C=10pt@R=10pt{
\CM\circ A \ar[r] & (\CC\cco\CM)\circ A \ar[r] & (\CO\cco\CM)\circ A \ar@{}[r]|-{\cong} \ar[d] & \CO\circ(A;\CI)\otimes(\CM\circ A) \ar[d]\\
&& (\CO\cco(\CC\cco\CM))\circ A \ar@{}[r]|-{\cong}\ar[d] & \CO\circ(A;\CI)\otimes(\CC\cco\CM)\circ A\ar[d]\\
&& (\CO\cco(\CO\cco\CM))\circ A \ar@{}[r]|-{\cong}\ar[dd] & \CO\circ(A;\CI)\otimes(\CO\cco\CM)\circ A\ar[d]\\
&&& \hspace{-1cm}\CO\circ(A;\CI)\otimes\CO\circ(A;\CI)\otimes(\CM\circ A)\ar[d]^{\mu_{\CO\circ(A;\CI)}}\\
&& (\CO\cco\CM)\circ A \ar@{}[r]|-{\cong} & \CO\circ(A;\CI)\otimes(\CM\circ A)
}\]
The top row is the map which defines the differential on $\CO\circ(A;\CI)\otimes(\CM\circ A)$, while the right column is this same map with $\CO\circ(A;\CI)\otimes(-)$ applied, followed by the algebra multiplication map.
The rightmost horizontal arrows are all isomorphisms, compatible with the vertical maps.
So to show that $d^l$ squares to~0 we need only consider the first two maps on the top row followed by the left-most vertical maps.
But note that this is just the functor $(-)\circ A$ applied to the sequence
\begin{equation}\label{eq_COCMdiff}\xymatrix{
\CM \ar[r] & (\CC\cco\CM) \ar[r] & (\CO\cco\CM) \ar[dll] \\
\CO\cco(\CC\cco\CM) \ar[r] & \CO\cco(\CO\cco\CM) \ar[r] & \CO\cco\CM.
}\end{equation}
To see that this composite is zero we consider the \emph{right twisted composite product} $\CO\circ_\kappa\!\CC$ defined in~\cite{Loday2012}.
This has underlying complex $\CO\circ\CC$ with a differential defined similarly to the left twisted composite product~\eqref{eq_ltcp}.
The first two arrows of~\eqref{eq_COCMdiff} define the differential of a complex $\CO\!\circ_\kappa\!\CM$ given by replacing $\CC$ with $\CM$ in an analogous manner to Lemma~\ref{lemma_barlemma}.
Thus the composite of all arrows in~\eqref{eq_COCMdiff} is~0 and we have shown that \((d^l)^2=0\).
The compatibility of \(d^l\) and \(\id\otimes\, d_\kappa\) follows from compatibility of left and right comodule structures on \(\CM\).
So $\LLO^\CM A$ is a chain complex.

To see that $\LLO^{(-)} A$ defines a right exact functor we use the fact that as a graded vector space it is isomorphic to $\UO A \otimes_k ((-)\circ_\kappa\!A)$.
By Lemma~\ref{lemma_barlemma} the functor $(-)\circ_\kappa\!A$ is right exact and the underlying graded vector space is given by tensoring by a vector space over the base field, hence $\LLO^{(-)}A$ is also right exact.

When the characteristic is 0, use Lemma~\ref{lemma_barlemma} again to see that $(-)\circ_\kappa\!A$ is left exact as well and again we are tensoring with $\UO A$ over the basefield, so $\LLO^{(-)}A$ is also left exact.
\end{proof}

When $A$ is a free algebra, $\LLO^\CM A$ can be computed in a special case:
\begin{lemma}\label{lemma_LLOfreeA}
Let $\CO$ be a Koszul operad and $A=\CO\circ V$ a free $\CO$-algebra generated by a chain complex~\(V\).
Let $\CM$ be a $\CO^\antis$-comodule such that 
\[
\CM \cong \CB \circ \CO^\antis
\]
as a right $\CO^\antis$-comodule with $\CB$ an $\sym$-module concentrated in a single arity, $r$.
Then the cotangent complex $\LLO^\CM A$ of $A$ associated to $\CM$ is quasi-isomorphic to the free $\UO A$-module
\[
\UO A \otimes \CB \circ V.
\] 
\end{lemma}
\begin{proof}
The cotangent complex $\LLO^\CM A$ is defined to be $\UO A \otimes \CM\circ_\kappa A$ with an added differential $d_l$.
We filter this complex using the weight of $V$ in $\CM\circ A \cong \CM\circ \CO \circ V$, precisely
\[
F_p\LLO^\CM A := \UO A \otimes \bigoplus_{k\leq p}(\CM\circ\CO)(k)\circ V.
\]
This is a filtration of complexes because the differential associated to $\CM\circ_\kappa\CO$ preserves the arity and hence the weight of $V$ and the added differential $d_l$ properly reduces the weight of $\CM\circ A$.
Therefore the associated graded module is isomorphic to
\[
\UO A \otimes \CM \circ_\kappa A.
\]
By Lemma~\ref{lemma_barlemma} this is in turn isomorphic to $\UO A \otimes \CB \circ (\CO^\antis\circ_\kappa A)$ which is quasi-isomorphic to $\UO A\otimes \CB \circ V$ since $A$ is freely generated by $V$ and \(\CO\) is Koszul.
Note that this is concentrated in a single weight $r$ of $V$ because $\CB$ is concentrated in arity $r$.

The spectral sequence associated to the above filtration has the form 
\[
E^1_{p\bullet} = \begin{cases}
\UO A \otimes \CB(r) \circ V & \text{ if $p=r$ and}\\
0 & \text{ otherwise.}
\end{cases}
\]
Hence the $E^1$ page is concentrated in a single column and so the spectral sequence collapses and we are done.
\end{proof}

\subsection{The main lemma}
\label{section_mainlemma}
We begin with our relative setup again.
Let $f:\CO\rightarrow \CP$ be a morphism of Koszul operads (taking generators to generators) and let \mbox{$g:\CC\rightarrow\CD$} be the corresponding morphism of Koszul dual cooperads.  Then there is a commuting square
\begin{equation}\label{eq_OPCDsquare}
\xymatrix{
\CO \ar[r]^f & \CP  \\
\CC \ar[u]^\kappa \ar[r]_g & \CD \ar[u]_(.54){\kappa'},
}\end{equation}
where the vertical arrows are twisting morphisms.
Let $\CM$ be a $\CC$-comodule.  Then it is naturally a $\CD$-comodule denoted $g^\ast \CM$.
Let $A$ be a $\CP$-algebra; the $\CO$-algebra is denoted $f^\ast A$.

\begin{lemma}\label{lemma_cotangentlemma}
There is an isomorphism of $\UP A$-modules
\[
\UP A \otimes_{\UO\!f^\ast\!A} \LLO^\CM f^\ast A \;\cong\; \LLP^{g^\ast \CM} A.
\]
\end{lemma}
\begin{proof}
The $\UO f^\ast A$-module $\LLO^\CM f^\ast A$ is based on the free module \mbox{$\UO\! f^\ast\! A\otimes\!(\CM\!\circ_\kappa\!\! f^\ast\! A)$}, with an added differential.
Hence $\UP A \otimes_{\UO\!f^\ast\!A} \LLO^\CM f^\ast A$ is based on the free $\UP A$-module 
\[
\UP A \otimes (\CM\!\circ_\kappa\! f^\ast A),
\]
with differential induced by 
\begin{equation}\label{eq_map1}\xymatrix{
\CM\circ_\kappa\! f^\ast A \ar[r] & \UO f^\ast A \otimes (\CM\circ_\kappa\! f^\ast A) \ar[r] & \UP A \otimes (\CM\circ_\kappa\! f^\ast A),
}\end{equation}
where the left arrow is given by~\eqref{eq_Mcotangent} and the right arrow is given by $\Uf A \otimes \id$.

Similarly the $\UP A$-module $\LLP^{g^\ast\!\CM}\!A$ is based on the free module \mbox{$\UP\! A\! \otimes\! (g^\ast\!\CM\!\circ_{\kappa'}\! A)$} with differential given by the extension of a map
\begin{equation}\label{eq_map2}\xymatrix{
g^\ast\!\CM\circ_{\kappa'} A \ar[r] & \UP A \otimes (g^\ast\! \CM\circ_{\kappa'}\! A),
}\end{equation}
again from~\eqref{eq_Mcotangent}.
By Lemma~\ref{lemma_barlemma}, there is an isomorphism of chain complexes
\[
\CM\circ_{\kappa}\! f^\ast\! A \;\cong\; g^\ast\!\CM\circ_{\kappa'}\! A.
\]
So it remains to prove that the maps~\eqref{eq_map1} and~\eqref{eq_map2} are equal.
Consider the diagram 
\[\xymatrix@C=50pt{
\CM\circ A  \ar[dr]^{\Delta^L_{g^\ast \CM}\circ \id} \ar[d]_{\Delta^L_\CM\circ \id} & \\
(\CC\cco\CM)\circ A \ar[r]_{(g\cco\id)\circ\id}\ar[d]_\cong  &  (\CD\cco\CM)\circ A \ar[d]^\cong \\
\CC\circ(A;\CI)\otimes(\CM\circ A) \ar[r]_{g\circ(\id;\id)\otimes\id}\ar[d]_{\kappa\circ(\id;\id)\otimes\id}  &  \CD\circ(A;\CI)\otimes(\CM\circ A) \ar[d]^{\kappa'\circ(\id;\id)\otimes\id} \\
\CO\circ(A;\CI)\otimes(\CM\circ A) \ar[r]_{f\circ(\id;\id)\otimes\id}\ar[d]_{\pi_A\otimes\id}  &  \CP\circ(A;\CI)\otimes(\CM\circ A) \ar[d]^{\pi_A\otimes\id} \\
\UO f^\ast\! A \otimes (\CM\circ A) \ar[r]_{\Uf A\otimes\id}  &  \UP A \otimes (\CM\circ A).
}\]
The left edge is precisely~\eqref{eq_Mcotangent} applied to the $\CC$-comodule $\CM$ and the $\CO$-algebra $f^\ast A$; so with the bottom edge we get~\eqref{eq_map1}.
Meanwhile the top diagonal edge and right edges are given by~\eqref{eq_Mcotangent} applied to the $\CD$-comodule~$g^\ast \CM$ and the $\CP$-algebra~$A$ and so are equal to~\eqref{eq_map2}.
To finish we show that each square commutes.

The top square consists of vertical isomorphisms which are compatible with $g$ and so this square commutes.
The middle square is given by~\eqref{eq_OPCDsquare} with \mbox{$(-)\circ(\id;\id)\otimes\id$} applied and so this square commutes.
The final square is given by applying \mbox{$(-)\otimes\id$} to a square in~\eqref{eq_UfA} which is used to define $\Uf A$ and so commutes.
\end{proof}

\subsection{Transfer of structure from cooperad comodules to (co)homology theories}
We saw in Section~\ref{section_cotangentcomplex} that the generalised cotangent complex $\LLO^{(-)}A$ is a right exact functor from $\CC$-comodules to left $\UO A$-modules.
Then in Section~\ref{section_mainlemma} we saw that the cotangent complex of a $\CP$-algebra viewed as an $\CO$-algebra and after a change of coefficients to $\UP A$ is an image of that functor.
Precisely it is the image of $\CC$ viewed as a $\CD$-comodule.
Hence structure in the category of $\CD$-comodules is reflected in the cotangent complex.
The next theorem makes use of this fact in a number of useful situations.
\begin{theorem}
\label{thm_cohomologyhodge}
Let $f:\CO\rightarrow\CP$ be a map of Koszul operads with $g:\CC\rightarrow\CD$ the corresponding map of Koszul dual cooperads.
Let $A$ be a $\CP$-algebra and $f^\ast A$ be its associated $\CO$-algebra.
Furthermore let $M$ be a left $\UP A$-module and $N$ be a right $\UP A$-module, we write $f^\ast M$ and $f^\ast N$ for their respective $\UO f^\ast A$-module structures.
\begin{enumerate}
\item[(I)] We have the following expressions for the homology and cohomology of $f^\ast A$ with coefficients in $f^\ast\! N$ and $f^\ast\! M$ respectively,
\begin{equation}\label{eq_Thomology}
H^\CO_\ast(f^\ast A, f^\ast N) \cong H_\ast(N\otimes_{\UP A} \LLP^{g^\ast \CC} A) 
\end{equation}
and
\begin{equation}\label{eq_Tcohomology}
H_\CO^\ast(f^\ast A, f^\ast M) \cong H^\ast(\Hom_{\UP A}(\LLP^{g^\ast \CC} A, M)),
\end{equation}
where $g^\ast \CC$ is the cooperad $\CC$ viewed as a $\CD$-comodule.
\item[(II)]
Suppose further that $g^\ast\CC$ decomposes as a $\CD$-comodule into a direct sum, $g^\ast\CC\cong \bigoplus_{i\in I}\CM_i$.
Then the homology decomposes
\begin{equation}\label{eq_Thomologydecomp}
H^\CO_\ast(f^\ast A, f^\ast N) \cong \bigoplus_{i\in I} H_\ast(N\otimes_{\UP A} \LLP^{\CM_i} A)
\end{equation}
and similarly for the cohomology
\begin{equation}\label{eq_Tcohomologydecomp}
H_\CO^\ast(f^\ast A, f^\ast M) \cong \prod_{i\in I} H^\ast(\Hom_{\UP A}(\LLP^{\CM_i} A, M)).
\end{equation}
\item[(III)]
  Suppose now that $A$ is concentrated in non-negative degrees and that $g^\ast\CC$ possesses a filtration, $\cbrac{\CM_i}_{i\in \mathbb{Z}}$ of $\CD$-comodules
\begin{equation}\label{eq_filtration}\xymatrix{
    \ldots \ar[r]^{a_{i-2}} & \CM_{i-1} \ar[r]^{a_{i-1}} & \CM_i \ar[r]^{a_{i+1}} & \CM_{i+1} \ar[r]^{a_{i+2}} & \ldots
}\end{equation}
such that 
\begin{enumerate}
\item each map $a_i$ is injective and their colimit is $g^\ast\CC$,
\item each quotient comodule $\CM_i/\CM_{i-1}$ is contained in differential degree~$|i|$ and greater.
\end{enumerate}
Then if $N$ is non-negatively graded there exists a spectral sequence converging to the homology $H^\CO_\ast\!(f^\ast\! A, f^\ast\! N)$ with 
\begin{equation}\label{eq_homSS}
  E^1_{p\bullet} = H_{\bullet-p}(N\otimes_{\UP A}\LLP^{\CM_{p} / \CM_{p-1}}A)  
\end{equation}
and if $M$ is non-positively graded then there exists a spectral sequence converging to the cohomology $H_\CO^\ast(f^\ast\! A, f^\ast\! M)$ with 
\begin{equation}\label{eq_cohomSS}
E_1^{p\bullet} = H^{\bullet-p}(\Hom_{\UP A}(\LLP^{\CM_{p} / \CM_{p-1}}A, M)).  
\end{equation}
\end{enumerate} 
\end{theorem}
\begin{proof}
\begin{enumerate}
\item[(I)]
For the homology, 
\begin{align*}
H_\ast^{\CO}(f^\ast A, f^\ast N) &= H_\ast(f^\ast N \otimes_{\UO f^\ast\!A} \LLO f^\ast\! A) \\
&\cong H_\ast(N \otimes_{\UP A} (\UP A \otimes_{\UO f^\ast\!A} \LLO f^\ast\! A)) \\
&\cong H_\ast(N \otimes_{\UP A} \LLP^{g^\ast \CC} A),
\end{align*}
where the first line is the definition of the homology groups, the second is a standard manipulation of tensor products and the final line is by Lemma~\ref{lemma_cotangentlemma}.
The cohomology follows a similar argument:
\begin{align*}
H_\CO^\ast(f^\ast A, f^\ast M) &= H^\ast(\Hom_{\UO f^\ast\! A}(\LLO f^\ast A, f^\ast M)) \\
&\cong H^\ast(\Hom_{\UP A}(\UP A \otimes_{\UO f^\ast\! A} \LLO f^\ast A, M)) \\
&\cong H^\ast(\Hom_{\UP A}(\LLP^{g^\ast\CC} A, M)),
\end{align*}
where again the first line is the definition of the cohomology groups, the second line now uses an adjunction involving change of coefficients and the final line is by Lemma~\ref{lemma_cotangentlemma} again.
\item[(II)]
By Lemma~\ref{lemma_LLM} the functor $\LLP^{(-)}A$ preserves coproducts so
\[
\LLP^{g^\ast\!\CM}A \;\cong\; \bigoplus_{i\in I} \LLP^{\CM_i}A.
\]
Using this identity in equations~\eqref{eq_Thomology} and~\eqref{eq_Tcohomology} from part~(I) we get~\eqref{eq_Thomologydecomp} and~\eqref{eq_Tcohomologydecomp} respectively.
\item[(III)]
Applying the functor $\LLP^{(-)}A$ to the filtration~\eqref{eq_filtration} gives a filtration of $\UP A$-modules.
\[\xymatrix{
  \ldots \ar[r] & \LLP^{\CM_{i-1}}A \ar[r] &\LLP^{\CM_i}A \ar[r] &\LLP^{\CM_{i+1}}A \ar[r] & \ldots
}\]
As the functor $\LLP^{(-)}A$ is right exact it preserves cokernels and so
\begin{equation}\label{eq_LLquotient}
  \LLP^{M_{p}}\!A \;/\; \LLP^{M_{p-1}}\!A \;\cong\; \LLP^{\CM_p / \CM_{p-1}}\!A.
\end{equation}
The union of this filtration is $\LLP^{g^\ast\CC}\!A$ and the non-negative grading of $A$ along with the degree conditions on $\CM_i$ imply that the $\UP A$-module $\LLP^{\CM_i/\CM_{i-1}} \! A$ is concentrated in degree $|i|$ and above.

We now apply the functors $N\otimes_{\UP A}(-)$ and $\Hom_{\UP A}(-,M)$.
In the first case the functor $N\otimes_{\UP A}\!(-)$ gives a filtration of chain complexes
\[\xymatrix{
  \ldots \ar[r] & N\otimes_{\UP A}\!\LLP^{\CM_i}\!A \ar[r] &N\otimes_{\UP A}\!\LLP^{\CM_{i+1}}\!A \ar[r] & \ldots
}\]
Since the functor is right exact the quotients are of the correct form, $N\otimes_{\UP A}\!\LLP^{\CM_p / \CM_{p-1}}\!A$ and the degree conditions on $N$ and $\LLP^{\CM_i/\CM_{i-1}}\!A$ imply that the spectral sequence does converge to the homology with given first page~\eqref{eq_homSS}.

In the second case we obtain a cofiltration
\[\xymatrix{
\ldots &\Hom_{\UP A}(\LLP^{\CM_{i-1}}A, M) \ar[l] &\Hom_{\UP A}(\LLP^{\CM_{i}}A, M) \ar[l] & \ldots \ar[l]
}\]
The first page of the cohomology spectral sequence has $E_1^{p\bullet} = H^{\bullet-p}(\ker(b_p))$, and 
\[
\ker(b_p) \cong \Hom_{\UP A}(\LLP^{\CM_p}\!A \,/\, \LLP^{\CM_{p-1}}\!A, M)
\cong \Hom_{\UP A}(\LLP^{\CM_p / \CM_{p-1}}\! A, M),
\]
where the second isomorphism arises from~\eqref{eq_LLquotient}, this provides~\eqref{eq_cohomSS}.
The degree condition on $\LLP^{\CM_i/\CM_{i-1}}\!A$ and the fact that $M$ is non-positively graded imply that $\Hom_{\UP A}(\LLP^{\CM_{i}/\CM_{i-1}}A, M)$ is concentrated in (cohomological) degrees $|i|$ and above. 
The limit of the cofiltration is $\Hom_{\UP A}(\LLP^{g^\ast\!\CC}\!A,M)$.
The degree condition implies that the associated spectral sequence converges as expected.
\end{enumerate}
\end{proof}

\begin{example}
As discussed in Remark~\ref{rem_HH} the cohomology theory of an associative algebra largely agrees with the Hochschild cohomology, similarly the homology theory largely agrees with the Hochschild homology.
The Koszul dual of the projection from $\As$ to $\Com$ is the inclusion of $\Lie$ into $\As$.
In Sections~\ref{section_fdls} and~\ref{section_equivalence} we will see that the associative operad viewed as a bimodule over the Lie operad has a filtration and that in characteristic~0 this filtration splits.

So in characteristic~0, part~(II) of the above theorem gives a decomposition of the Hochschild (co)homology of a commutative algebra $A$ with coefficients in some symmetric bimodule.
We will show in Section~\ref{section_equivalence} that this decomposition agrees with the classical decomposition~\cite{Gerstenhaber1987}.
\end{example}

\subsection{An application}
It was observed in~\cite{Koszul1950} that for an associative algebra $A$, the Hochschild complex of $A$ contains the Chevalley-Eilenberg complex of the Lie algebra $(A, [a,b]=ab-ba)$ as a sub-complex.
This generalises as follows.

Let $f:\CO\rightarrow\CP$ be a map of Koszul operads which induces an injection $g:\CC\rightarrow\CD$ on Koszul dual cooperads.  Let $A$ be a $\CP$-algebra with right $\UP A$-module $N$.
The map $g^\ast\CC\rightarrow \CD$ is an injection of $\CD$-comodules so we have an injection
\[
N\otimes_{\UP A} \LLP^{g^\ast\CC} A \hookrightarrow N\otimes_{\UP A} \LLP A.
\]
By Part~(I) of Theorem~\ref{thm_cohomologyhodge} the first complex computes the 
$\CO$-algebra homology of $f^\ast A$ with coefficients in $f^\ast N$ while the second complex computes the $\CP$-algebra homology of $A$ with coefficient in $N$.

Applying this to the map of operads $i:\Lie\rightarrow\As$ with an associative algebra $A$ and bimodule $M$ we find that the dual of $i$ is the inclusion $\Sigma\coCom\rightarrow\Sigma\coAs$ and there is an inclusion of complexes
\[
i^\ast M\otimes_{\UU_\Lie i^\ast A} \LL_\Lie i^\ast A \hookrightarrow M\otimes_{\UU_\As A} \LL_\As A,
\]
the first calculating the homology of the Lie algebra $i^\ast A$ with coefficients in $i^\ast M$ and the second the homology of the associative algebra $A$ with coefficients in $M$.
\begin{example}
Let $f:\PLie\rightarrow\Dend$ be the map of operads which encodes the fact that a dendriform algebra is naturally a pre-Lie algebra.  On Koszul dual cooperads we have an injection $\Sigma\coPerm\rightarrow\Sigma\coDias$.
Hence for a dendriform algebra $A$ there is a natural inclusion of the complex computing the homology of $A$ as a pre-Lie algebra into the complex computing the homology as a dendriform algebra.
This was described explicitly by Goichot~\cite{Goichot}.
\end{example}

\section{Operads with filtered distributive laws and their modules}
\label{section_fdls}
In the first half of this paper we related the (co)homology theories of algebras over a Koszul operad to comodules over the Koszul cooperad.
To apply these results in practice we require precise knowledge of the comodules themselves.
This knowledge may take the form of a filtration of comodules, leading to a spectral sequence computing the (co)homology.
We proceed in this section to look at a method for finding examples of such filtrations.

Rather than working with comodules for cooperads we will work with modules for operads simply because they are more familiar to calculate with.

\subsection{Filtrations of operads and their modules}\label{section_filtrations}
A \emph{filtration} of an operad $\CO$ is a series of sub-$\sym$-modules $(F_i\CO)_{i\in\mathbb{\mathbb{Z}}}\subseteq \CO$ such that
\begin{enumerate}
\item if $i\leq j$ then $F_i\CO \subseteq F_j\CO$ and
\item under the bilinear multiplication map, $\mu^1_\CO(F_i\CO\cco F_j\CO) \subseteq F_{i+j}\CO$.
\end{enumerate}
With this definition $F_0\CO$ is a suboperad of $\CO$ and has both a left and right action on each $F_s\CO$.
Thus the operad filtration is also a filtration of $F_0\CO$-bimodules.

The associated graded $\sym$-module $\gr_F\CO$ defined by $(\gr_F\CO)_s = F_s\CO / F_{s-1}\CO$ for $s\in\mathbb{Z}$ inherits an operad structure and we call $\gr_F\CO$ the \emph{associated graded operad}.
There is a morphism of operads
\[
F_0\CO \rightarrow F_0\CO / F_{-1}\CO \hookrightarrow \gr_F\CO
\]
and the inherited $F_0\CO$-bimodule structure on $\gr_F\CO$ agrees with the structure obtained by taking the associated graded of the $F_0\CO$-bimodule filtration of $\CO$.

\begin{example}[A positive filtration]
Let $g:\CQ\hookrightarrow\CO$ be an inclusion of operads and let $\CV$ be a sub-\(\sym\)-module of $\CO$, contained in arities 2 and greater, which together with $\CQ$ generates the whole operad $\CO$.
Define a filtration of $\CO$ by
\[
F_s\CO = \begin{cases}
0 & \text{ if $s < 0$, }\\
\CQ & \text{ if $s=0$ and}\\
\mu_\CO\bbrac{\CF_{\leq s}(\CQ\oplus \CV)} & \text{ if $s> 0$,}
\end{cases}\]
where $\CF_{\leq s}(\CQ\oplus\CV)$ is the sub-$\sym$-module of the free operad on $\CQ\oplus\CV$ spanned by terms containing no more than $s$ copies of $\CV$.
The fact that $\CF_{\leq s}(\CQ\oplus\CV)$ gives a filtration of the free operad on $\CQ\oplus\CV$ means that $F_s\CO$ is a filtration of $\CO$.
Since $F_0\CO = \CQ$ this operad filtration gives a filtration of $\CQ$-bimodules for $\CO$.
Along with $\CQ$ the $\sym$-module $\CV$ generates $\CO$ and so the filtration is full and locally finite in the sense that for each~$n$ there exists~$N$ such that~$m\geq N$ implies that $F_m\CO(n) = \CO(n)$.
\end{example}

\begin{example}[A negative filtration]
Let $f:\CO\rightarrow\CP$ be a surjection of operads, then the kernel $\CJ=\ker(f)$ is a two-sided ideal of $\CO$.
Define $\CJ^1 = \CJ$ and $\CJ^{i+1} = \mu_\CO(\CJ^i\cco\CJ)$; then $\CJ^i$ is the linear span of elements of $\CO$ obtainable by multiplying $i$ elements of $\CJ$.
Thus defining
\[
G_t\CO = \begin{cases}
\CJ^{-t} & \text{ if $t< 0$ and}\\
\CO & \text{ if $t\geq 0$,}
\end{cases}\]
we have an operad filtration of $\CO$.
It is also a filtration of bimodules for $G_0\CO = \CO$ and the ideal $G_{-1}\CO=\CJ$ acts trivially on the associated graded bimodule,~$\gr_G\CO$.
Hence the $\CO$-bimodule structure is inherited from a $\CP$-bimodule structure given equivalently by the inclusion of operads $\CP=(\gr_G\CO)_0\hookrightarrow\gr_G\CO$.

If $\CO$ and $\CP$ are generated by binary operations then $\CJ=\ker(f)$ is contained in arities~$2$ and greater.
Therefore $\CJ^{i}$ is contained in arities $i+1$ and greater and so
\[
\bigcap_{t\in\mathbb{Z}} G_t\CO = \bigcap_{i\geq 1} \CJ^i = 0.
\]
Also the filtration is locally finite in the sense that for each~$n$ there exists~$N$ such that~$m\leq N$ implies that $G_m\CO(n) = 0$.
\end{example}

\subsection{Filtered distributive laws}
In this section we will recall the definition of an operad with a filtered distributive law.
For any such operad we define both a positive and a negative filtration, obtaining two results about the bimodule structure.
These will be used in the following section to give two spectral sequences for (co)homology theories over these operads.

Distributive laws between quadratic operads were introduced in~\cite{Fox1997} and have since been generalised in a number of stages. 
The filtered distributive laws we use were defined in~\cite{Dotsenko2011fdls}, where they were applied to prove Koszulness of a number of operads, some of which we revisit below.

We start with the definition of a filtered distributive law.
Let $\CP=\CF(\CV)/(\CR)$ and $\CQ=\CF(\CW)/(\CS)$ be two quadratic operads and let
\begin{align}
 s: \CR \rightarrow & (\CW\cco\CV) \oplus (\CV \cco \CW) \oplus (\CW\cco\CW) \\
 d: (\CW\cco\CV) \rightarrow & (\CV\cco\CW) \oplus (\CW\cco\CW)
\end{align}
be maps of $\sym$-modules.
Define
\begin{align}
 \CT = & \cbrac{x - s(x) \mid x\in \CR} \subseteq (\CV\oplus\CW)\cco(\CV\oplus\CW)\\
 \CD = & \cbrac{x - d(x) \mid x\in \CW\cco\CV} \subseteq (\CV\oplus\CW)\cco(\CV\oplus\CW).
\end{align}
Then we define $\CO$ to be the quadratic operad with generators $\CV\oplus \CW$ and relations $\CT\oplus\CD\oplus\CS$.
Since $\CS$ is included in the space of relations it is clear that there is a map $g:\CQ\rightarrow\CO$ given by inclusion of generators.  
There is also a projection~$f:\CO\rightarrow\CP$ given by sending the generators~$\CW$ to~0.  To see that this is well defined, observe that the map $\CF(\CV\oplus\CW)\rightarrow\CF(\CV)$ sends~$\CT$ to~$\CR$ while~$\CD$ and~$\CS$ are both mapped to~$0$.

Since the relations arising from $\CD$ allow any element in $\CW\circ\CV$ to be rewritten in $(\CV\circ\CW)\oplus(\CW\circ\CW)$, the composite map 
\begin{equation}\label{eq_zetacomposite}
\CF(\CV) \circ \CF(\CW)\hookrightarrow \CF(\CV\oplus\CW) \rightarrow \CO
\end{equation}
is a surjection.  If the projection of operads $f:\CO\rightarrow\CP$ splits in the category of $\sym$-modules then we may define the map
\begin{equation}
\zeta:\CP\circ \CQ \rightarrow \CO,
\end{equation}
which is also a surjection.
\begin{definition}
The 4-tuple $(\CP,\CQ,s,d)$ constitutes a \emph{filtered distributive law} for~$\CO$ if the map $f:\CO\rightarrow\CP$ splits in the category of $\sym$-modules and if $\zeta$ induces an isomorphism when restricted to the weight~$3$ elements of $\CP\circ\CQ$.
\end{definition}
In Theorem~5.1 of~\cite{Dotsenko2011fdls} it was shown that under these conditions the whole map~$\zeta$ is in fact an isomorphism of $\sym$-modules.
The following strengthening of this theorem to consider the right $\CQ$-module structure of $\CO$ is immediate.
\begin{corollary}\label{cor_rightfilt}
Consider $\CP\circ\CQ$ as the free right $\CQ$-module with generators $\CP$.  Then if $\CO$ has a filtered distributive law the map $\zeta$ is an isomorphism of right $\CQ$-modules.
\end{corollary}

\begin{remark}
When the map $s$ is $0$ and the map $d'$ has image in $\CV\cco\CW$ then a filtered distributive law is equivalent to a distributive law as defined in~\cite{Fox1997}.
\end{remark}

\subsection{The associated graded operad from a positive filtration}
We will now assume that both $\CV$ and $\CW$ are concentrated in arity~$2$ so that $\CP$ and $\CQ$ are binary quadratic operads.

Using the identification of $\CO$ with $\CP\circ\CQ$, define a filtration of $\sym$-modules by
\[
F_s\CO = \bigoplus_{0< p\leq s+1}\CP(p)\circ\CQ.
\]
\begin{lemma}\label{lemma_assocgradedoperad}
The filtration of $\sym$-modules $(F_s\CO)_{s\in\mathbb{Z}}$ is an operad filtration of~$\CO$.
The associated graded operad has the distributive law~$(\CP,\CQ, 0, d')$ where $d'$ is the corestriction of $d$ to $\CV\cco\CW$.
\end{lemma}
\begin{proof}
The free operad $\CF(\CV\oplus\CW)$ is filtered by $\CF_{\leq s}(\CV\oplus\CW)$, the sub-$\sym$-module spanned by elements with no more than $s$ copies of $\CV$.
Such elements may be rewritten using the map $d$ to lie in $\CF_{\leq s}(\CV)\circ \CF(\CW)$ and under the map~\eqref{eq_zetacomposite} this surjects onto $F_s\CO$.

Now we calculate the associated graded operad by showing that the arity~2 operations generate an operad satisfying a distributive law.
In arity~2 we have
\[
\gr_F\CO(2) = F_0\CO(2) \oplus (F_1\CO/F_0\CO)(2) \cong \CQ(2) \oplus \CP(2) \cong \CW \oplus \CV.
\]
And in arity~3,
\begin{align*}
\gr_F\CO(3) &= F_0\CO(3) \oplus (F_1\CO/F_0\CO)(3) \oplus (F_2\CO/F_1\CO)(3) \\
&\cong\CQ(3) \oplus (\CV\cco\CW) \oplus \CP(3).
\end{align*}
From this we may read off the relations by noting that the image of two copies of $F_0\CO(2)$ must lie in $F_0\CO(3)$; 
the image of one copy of $F_0\CO(2)$ and one copy of $F_1\CO(2)/F_0\CO(2)$ must lie in $F_1\CO(3)/F_0\CO(3)$; and the image of two copies of $F_1\CO(2)/F_0\CO(2)$ must lie in $F_2\CO(3)/F_1\CO(3)$.
Thus we find the relations from $\CQ$ in the first case, relations given by the corestriction of $d$, denoted $d'$ from the second case and the relations from $\CP$ in the third case.
That $d'$ defines a distributive law is immediate because we may check that $\gr_F\CO(4)$ is isomorphic to $\CP\circ\CQ(4)$.
We know that there can be no other relations because the operad defined by~$(\CP,\CQ,0,d')$ is isomorphic to $\CP\circ\CQ$ as an $\sym$-module which is also the underlying $\sym$-module of~$\gr_F\CO$.
\end{proof}

\begin{definition}
We say that a filtered distributive law $(\CP,\CQ,s,d)$ is \emph{semi-filtered} if~$d$ has image in $\CV\circ\CW$.
\end{definition}

\begin{corollary}\label{cor_posfilt}
There is a positive filtration of $\CO$ as a $\CQ$-bimodule with graded pieces $\CP(p)\circ\CQ$ which are free as a right module with left module structure provided by the corestriction of $d$ to $\CV\cco\CW$.
Furthermore if $(\CP,\CQ,d,s)$ is semi-filtered then the filtration of $\CQ$-bimodules splits and so~$\CO$ is isomorphic to~$\gr_F\CO$ as a~$\CQ$-bimodule.
\end{corollary}
\begin{proof}
The first part is immediate from Lemma~\ref{lemma_assocgradedoperad} and the discussions of Section~\ref{section_filtrations}.
When the law is semi-filtered $d$ already has image in $\CV\cco\CW$ and so taking the corestriction does not change $d$.
The fact that both operads $\CO$ and $\gr_F\CO$ are isomorphic to $\CP\circ\CQ$ as $\sym$-modules gives the isomorphism of $\CQ$-bimodules.
\end{proof}
\begin{example}\label{ex_aslie}
In characteristic~$0$ the associative operad has a semi-filtered distributive law $(\Com,\Lie,s,d)$ meaning that as a $\Lie$-bimodule it decomposes into pieces $i^\ast \As \cong \oplus_{p\geq 0}\Com(p) \circ \Lie$.
The associated graded operad is the operad describing Poisson algebras.
\end{example}

\subsection{The associated graded operad from a negative filtration}
Define another $\sym$-module filtration by
\[
G_t\CO(n) = \begin{cases}
\CO(n) & \text{ if $t\geq 0$ and}\\
\bigoplus_{0 < p \leq n+t}\bbrac{\CP(p)\circ\CQ}(n) & \text{ if $t<0$.}
\end{cases}\]
This may be interpreted as saying that for elements in $G_{-s}\CO$ the weight contributed by $\CQ$ must not be less than $s$.
\begin{lemma}
The filtration of $\sym$-modules $(G_s\CO)_{s\in\mathbb{Z}}$ defines an operad filtration of~$\CO$.
The associated graded operad has the distributive law~$(\CP,\CQ, 0, d')$ where~$d'$ is the corestriction of~$d$ to~$\CV\cco\CW$.
\end{lemma}
\begin{proof}
The free operad $\CF(\CV\oplus\CW)$ is negatively filtered by $\CF_{\geq t}(\CV\oplus\CW)$, the sub-$\sym$-module spanned by elements with no less than $t$ copies of $\CW$.
Using the maps $d$ such elements may be rewritten to lie in $\CF(\CV)\circ\CF_{\geq t}(\CW)$ since $d$ can only increase the number of copies of $\CW$.
Under the map~\eqref{eq_zetacomposite} this term of the filtration surjects onto~$G_{-t}\CO$.

Again we look at $\gr_G\CO$ in arities $2$ and $3$.
We find that
\[
\gr_G\CO(2) = (G_0\CO / G_{-1}\CO)(2) \oplus (G_{-1}\CO/G_{-2}\CO)(2) \cong \CV\oplus\CW
\]
and
\begin{align*}
\gr_G\CO(3) &= (G_0\CO / G_{-1}\CO)(3) \oplus (G_{-1}\CO/G_{-2}\CO)(3) \oplus (G_{-2}\CO/G_{-3}\CO)(3) \\
&\cong \CP(3) \oplus (\CV\cco\CW) \oplus \CQ(3).
\end{align*}
The image of the operad multiplication on two copies of $(\gr_G\CO)_0(2)$ must lie in~$(\gr_G\CO)_0(3)$;
the image of one copy of~$(\gr_G\CO)_0(2)$ and one copy of~$(\gr_G\CO)_{-1}(2)$ must lie in~$(\gr_G\CO)_{-1}(3)$;
while two copies of~$(\gr_G\CO)_{-1}(2)$ are taken to~$(\gr_G\CO)_{-2}(3)$.
Thus we see that~$\gr_G\CO$ is generated by~$\CV\oplus\CW$ with relations including those of~$\CP$, those of~$\CQ$ along with a distributive law provided by~$d'$.
That this is a well-defined distributive law is checked by looking at~$(\gr_G\CO)(4)\cong (\CP\circ\CQ)(4)$ and that there are no further relations by noting that~$(\gr_G\CO)(n) \cong (\CP\circ\CQ)(n)$.
\end{proof}
The positive and negative filtrations give associated graded operads which are isomorphic as ungraded operads.
However not only the gradings disagree; the results for modules are very different.
\begin{corollary}\label{cor_negfilt}
There is a filtration of $\CO$ as a bimodule over itself with graded pieces $(\gr_G\CO)_{-t}(n) = \bbrac{\CP(n-t)\circ\CQ}(n)$ on which the suboperad $\CQ$ acts trivially.
Thus the left and right $\CO$-actions are given by the pullback of a $\CP$-bimodule structure along $f:\CO\rightarrow\CP$.
Over $\CP$ the associative graded bimodule, $\gr_G\CO$ is free as a left module and has the right module structure provided by the corestriction of $d$ to $\CV\cco\CW$.
\end{corollary}
\begin{proof}
By the discussion in Section~\ref{section_filtrations} the operad filtration is also a filtration of bimodules for $G_0\CO\cong \CO$.
The action of $\CO$ on the associated graded bimodule is inherited from the action of the suboperad $(\gr_G\CO)_0$ on $\gr_G\CO$.
Since $G_{-1}\CO$ is the kernel of $f:\CO\rightarrow\CP$ we find that the action of $\CQ$ on $\gr_G\CO$ is trivial.

The graded piece $(\gr_G\CO)_{t}$ consists of those pieces of $\CP\circ\CQ$ with weight~$-t$ contributed by the $\CQ$ terms and with no restriction on $\CP$, thus is free as a left $\CP$-module, while the right action comes from the distributive law~$(\CP,\CQ, 0, d')$ for $\gr_G\CO$ where $d'$ is the corestriction of $d$ to $\CV\cco\CW$.
\end{proof}

\begin{example}
In Example~\ref{ex_aslie} the positive filtration was used to calculate the $\Lie$-bimodule structure of the associative operad in characteristic~$0$ when it had the filtered distributive law~$(\Com,\Lie,s,d)$. 
With the negative filtration one has a filtration of $\As$ as a bimodule over itself and the associative graded bimodule is isomorphic to the operad for Poisson algebras with an action of $\As$ given by the embedding of the commutative operad, $\Com$.
\end{example}

\subsection{Applications to cohomology theories}
In this section we use the operadic module theory just discussed to study (co)homology theories by applying Theorem~\ref{thm_cohomologyhodge}.
In~\cite{Dotsenko2011fdls} it was shown that if $\CO$ has the filtered distributive law~$(\CP,\CQ,s,d)$ for Koszul operads~$\CP$ and~$\CQ$ then~$\CO$ is Koszul and the Koszul dual operad~$\CO^!$ has the law~$(\CQ^!,\CP^!,s^!,d^!)$.

For the rest of the section we will assume that all operads are binary, quadratic and Koszul and that the field~$k$ has characteristic~0.
Each of the corollaries~\ref{cor_rightfilt},~\ref{cor_posfilt} and~\ref{cor_negfilt} has an analogue for cooperads proved either by repeating the proofs in the coalgebraic framework or by taking graded linear duals with the assumption that the generating $\sym$-modules~$\CV$ and~$\CW$ are finite dimensional in each arity.

Instead of the positive filtration of~\ref{cor_posfilt} one has a negative filtration with
\begin{equation}\label{eq_conegfilt}
F_t\CO^\antis = \begin{cases}
\CO^\antis & \text{ if $t \geq 0$ and}\\
\bigoplus_{p > -t}\CQ^\antis(p)\circ\CP^\antis & \text{ if $t < 0$.}
\end{cases}\end{equation}
To see how this filtration may be obtained from the positive filtration of the operad~$\CO^!$, we note that taking the linear dual of~$F_{t-1}\CO\hookrightarrow F_t\CO$ gives a surjection~$(F_{t-1}\CO)^\ast \leftarrow (F_{t}\CO)^\ast$ and a decreasing filtration.
To obtain an increasing filtration one takes the kernel of the map~$\pi_t:(\CO^!)^\ast \rightarrow (F_{t}\CO)^\ast$ for each~$t$.
This yields an injection~$\ker(\pi_t)\hookrightarrow\ker(\pi_{t-1})$ and to obtain the filtration of cooperads one now reindexes the $t$ to obtain a negative cooperad filtration.

In a similar manner, instead of the negative filtration of~\ref{cor_negfilt} one has a positive filtration with
\begin{equation}\label{eq_coposfilt}
G_t\CO^\antis(n) = \begin{cases}
\bigoplus_{p \geq n-t} \bbrac{\CQ^\antis(p)\circ\CP^\antis}(n) & \text{ if $t\geq 0$ and}\\
0 & \text{ if $t<0$.}
\end{cases}\end{equation}
Equivalently $G_t\CO^\antis$ is the subspace of~$\CO^\antis\cong\CQ^\antis\circ\CP^\antis$ with the weight provided by~$\CP^\antis$ no more than~$t$.
\begin{definition}
We say that a filtered distributive law $(\CP,\CQ,s,d)$ is \emph{op-semi-filtered} if the Koszul dual law $(\CQ^!, \CP^!, s', d')$ is semi-filtered.
\end{definition}

First we consider the bar homology.
\begin{proposition}
\label{prop_barhodge}
Let $\CO$ be an operad with filtered distributive law $(\CP, \CQ, s, d)$ and let $A$ be a $\CP$-algebra.
Then 
\begin{equation}
B_\CO(f^\ast A) \cong \CQ^\antis \circ B_\CP(A).
\end{equation}
\end{proposition}
\begin{proof}
We use the module theory of operads with filtered distributive laws applied to the Koszul dual operad~$\CO^!$.
By Corollary~\ref{cor_rightfilt} the right $\CP^!$-modules~$\CO^!$ and $\CQ^!\circ\CP^!$ are isomorphic.
Dually the right~$\CP^\antis$-comodule~$\CO^\antis$ and the free right $\CP^\antis$-comodule~$\CQ^\antis\circ\CP^\antis$ are also isomorphic.
Applying part (II) of Theorem~\ref{thm_barhodge} with $\CC=\CO^\antis$, $\CD=\CP^\antis$ and $\CB = \CQ^\antis$ gives the result.
\end{proof}
One application of this proposition is to compute the homology of the bar complex of a free $\CP$-algebra viewed as a $\CO$-algebra.
Let $A$ be the free $\CP$-algebra generated by a vector space $V$ then since $\CP$ is Koszul we have $H_\ast(B_\CP(A)) \cong V$ and the corollary shows that $H_\ast(B_\CO(A))\cong \CQ^\antis \circ V$.
Note that for this last isomorphism we require the assumption that~$k$ is of characteristic~$0$.

Now we treat the (co)homology of a $\CP$-algebra~$A$ viewed as an $\CO$-algebra.
Let $M$ be a left $\UP A$-module and $N$ a right $\UP A$-module and as usual denote by $f^\ast M$ and $f^\ast N$ the corresponding $\UO f^\ast A$-modules.
For this we use the positive filtration of~$\CO^!$ or equivalently the negative filtration~\eqref{eq_conegfilt} of~$\CO^\antis$.
The following theorem generalises Theorem~10.2 of~\cite{Fox1997}.
\begin{theorem}
\label{thm_posfiltss}
Suppose that a non-negatively graded operad $\CO$ satisfies a filtered distributive law $(\CP, \CQ, s, d)$ with $\CP$ and $\CQ$ Koszul operads and suppose that~$A$ is a~$\CP$-algebra concentrated in non-negative degrees.
\begin{enumerate}
\item[(I)]
For a right $\UP A$-module $N$ concentrated in non-negative degrees there is a spectral sequence converging to the homology $H^\CO_\ast(f^\ast\! A, f^\ast\! N)$ with
\[
E^1_{p\bullet} = \begin{cases}
H_{\bullet-p}\bbrac{N\otimes_{\UP A} \LLP^{\CQ^\antis(1-p)\circ\CP^\antis} A} & \text{ for $p\leq 0$,}\\
0 & \text{ otherwise.}
\end{cases}
\]
and for a left $\UP A$-module $M$ concentrated in non-positive degrees there is a spectral sequence converging to the cohomology $H_\CO^\ast(f^\ast\! A, f^\ast\! M)$ with
\[
E_1^{p\bullet} = \begin{cases}
H^{\bullet-p}\bbrac{\Hom_{\UP A}(\LLP^{\CQ^\antis(1-p)\circ\CP^\antis} A, M)} & \text{ for $p\leq 0$,}\\
0 & \text{ otherwise.}
\end{cases}
\]
\item[(II)]
If furthermore the distributive law is op-semi-filtered then the homology decomposes
\[
H^\CO_\ast(f^\ast A, f^\ast N) \cong \bigoplus_{n\geq 1} H_\ast\bbrac{N\otimes_{\UP A} \LLP^{\CQ^\antis(n)\circ \CP^\antis} A}
\]
and similarly for the cohomology
\[
H_\CO^\ast(f^\ast A, f^\ast M) \cong \bigoplus_{n\geq 1} H^\ast\bbrac{\Hom_{\UP A}(\LLP^{\CQ^\antis(n)\circ \CP^\antis} A, M)}.
\]
\end{enumerate}
\end{theorem}
\begin{proof}
We apply part~(III) of Theorem~\ref{thm_cohomologyhodge} to the filtration~\eqref{eq_conegfilt} of~$\CO^\antis$ provided by a coalgebraic analogue of Corollary~\ref{cor_posfilt} to find the spectral sequences.
We find that $F_{-t}\CO^\antis / F_{-t-1}\CO^\antis$ is isomorphic to $\CQ^\antis(t+1)\circ\CP^\antis$ for $t\leq 0$ and $0$ for $t > 0$.
Since $\CQ^\antis(t+1)$ has degree greater than or equal to $t$ and $\CP^\antis$ is non-negatively graded the condition on the $\CP^\antis$-comodule filtration holds and guarantees convergence.

In the op-semi-filtered case the Koszul dual~$\CO^!$ is semi-filtered and we apply the same corollary to find that the $\CP^\antis$-comodule $\CO^\antis$ decomposes into a direct sum, so applying part~(II) of Theorem~\ref{thm_cohomologyhodge} gives the homology case.
In the cohomology case the Theorem gives a decomposition as a direct product, however since $A$ is non-negatively graded the decomposition of $\LLP^{\CQ^\antis\circ\CP^\antis} A$ is degree-wise finite and so the decomposition as a direct sum is valid.
\end{proof}

When the $\CP$-algebra is free then whether op-semi-filtered or not both the homology and cohomology decompose and can be calculated explicitly.
\begin{theorem}\label{thmFree}
Let $\CO$ be an operad concentrated in degree 0 with a filtered distributive law $(\CP,\CQ,s,d)$ for Koszul operads $\CP$ and $\CQ$.  Let $A=\CP\circ V$ be a free $\CP$-algebra with generators concentrated in degree~$0$ and let $M$ be a left $\UP A$-module and $N$ a right $\UP A$-module.
Then the homology of $f^\ast A$ with coefficients in $f^\ast N$ is isomorphic to 
\[
N\otimes_k(\CQ^\antis\circ V)
\]
and the cohomology of $f^\ast A$ with coefficients in $f^\ast M$ is isomorphic to 
\[
\Hom_k(\CQ^\antis\circ V, M).
\]
\end{theorem}
\begin{proof}
We will show that there is a quasi-isomorphism of $A$-modules,
\[
\UP A \otimes_k \CQ^\antis\circ V \quis \UP A \otimes_{\UU_\CO f^\ast A} \LLO f^\ast A,
\]
from which the results immediately follow.
By Lemma~\ref{lemma_cotangentlemma} we have
\[
\UP A\otimes_{\UO f^\ast A} \LLO f^\ast A \;\cong\; \LLP^{g^\ast\CO^\antis}A,
\]
where $g^\ast\CO^\antis$ is the Koszul dual cooperad of $\CO$ viewed as a $\CP^\antis$-comodule.
This comodule contains $\CQ^\antis$ as a sub-$\CP^\antis$-comodule where both left and right coproducts are zero, in fact~$\CQ^\antis$ is the space of coindecomposables of $\CO^\antis$.
From this inclusion of comodules and the inclusion of the generators~$V$ in~$A$ there is a map
\[
  \UP A \otimes \CQ^\antis\circ V \,\cong\, \LLP^{\CQ^\antis}A \,\hookrightarrow\, \LLP^{g^\ast\CO^\antis}A
\]
and it remains to show that this is a quasi-isomorphism.
We use the filtration~\eqref{eq_conegfilt} of $g^\ast\CO^\antis$ provided by Corollary~\ref{cor_posfilt}.  
This is a negative filtration of $\CP^\antis$-comodules and yields a spectral sequence with 
\[
E^1_{p\bullet} = \begin{cases}
H_{\bullet-p}(\LLP^{F_p\CO^\antis/F_{p-1}\CO^\antis}A) & \text{ for $p\leq 0$ and}\\
0 & \text{ otherwise,}
\end{cases}
\]
where $F_p\CO^\antis/F_{p-1}\CO^\antis\cong\CQ^\antis(1-p)\circ \CP^\antis$ is free as a right $\CP^\antis$-comodule.
Since the generators are concentrated in arity $1-p$ and $A$ is a free $\CP$-algebra we may apply Lemma~\ref{lemma_LLOfreeA} to find
\[
E^1_{p\bullet} = \begin{cases}
s^{-p}\UP A \otimes \CQ^\antis(1-p) \circ V & \text{ for $p\leq 0$ and}\\
0 & \text{ otherwise.}
\end{cases}
\]
Since $\CQ^\antis(1-p)$ is of differential degree $-p$ with $V$ and $\UP A$ both in degree $0$ we find that the first page $E^1$ is $0$ except in bidegrees $(p, -2p)$.  This guarantees the collapse of the spectral sequence proving the quasi-isomorphism.
\end{proof}

\begin{remark}
The result above is true for free~$\CP$-algebras generated by any chain complex~$V$.
The proof is much the same, although instead of working with the free $\CP$-algebra $\CP\circ V$ we instead work with the free $\sym$-module-$\CP$-algebra $\CP\circ \CI \cong \CP$.
This requires using the enveloping algebra and cotangent complex in the category of $\sym$-modules but the proof above works to show that 
\[
\UP \CP \otimes_{\UO f^\ast \CP} \LLO f^\ast \CP \quis \UP\CP \otimes \CQ^\antis
\]
and then to obtain the result for a general space of generators one can apply the functor $(-\circ V)$.
\end{remark}

Next we transfer the module results found in Corollary~\ref{cor_negfilt} for the negative filtration of~$\CO^\antis$ to the~$\CO$-algebra (co)homology.
The associated spectral sequences apply to any~$\CO$-algebra, but the results acheived are less strong than those of Theorem~\ref{thm_posfiltss}; in fact if~$A$ takes the form~$f^\ast B$ for~$B$ a~$\CP$-algebra the differentials on the zeroth page are all zero.

Denote by $h$ the composite morphism
\[
\gr_G\CO \rightarrow \CQ \hookrightarrow \CO
\]
where the first morphism is given by sending $\CV$ to zero, which is well-defined for a distributive law.  
Then any~$\CO$-algebra~$A$ is an algebra~$h^\ast A$ for the associated graded operad~$\gr_G\CO$.
For Koszul dual cooperads we have~$(\gr_G\CO)^\antis \cong \gr_G\CO^\antis$ and will denote by $k$ the map
\[
\gr_G\CO^\antis \rightarrow\CQ^\antis\rightarrow\CO^\antis.
\]
In the following let~$M$ be a left~$\UO A$-module and~$N$ a right~$\UO A$-module.
\begin{theorem}\label{thm_negfiltss}
Suppose that non-negatively graded $\CO$ satisfies a filtered distributive law~$(\CP,\CQ,s,d)$ with~$\CP$ and~$\CQ$ Koszul operads.  
Then for a non-negatively graded~$\CO$-algebra~$A$ there are spectral sequences
\[
E_{p\bullet}^1 = H_{\bullet-p}\bbrac{N\otimes_{\UO A} \LLO^{(\gr_G\CO^\antis)_p}A} \quad\text{ and }\quad 
E_1^{p\bullet} = H^{\bullet-p}\bbrac{\Hom_{\UO A}(\LLO^{(\gr_G\CO^\antis)_p}A, M},
    \]
for $p\geq 0$ and $0$ otherwise.
They converge to~$H_\ast^\CO(A, N)$ and~$H^\ast_\CO(A,M)$ when $N$ is non-negatively graded and $M$ non-positively graded respectively.
There are isomorphisms
\[
E_{}^1 \cong H^{\gr_G\CO}_\ast(h^\ast A, h^\ast N)\quad\text{ and }\quad 
E^{}_1 \cong H_{\gr_G\CO}^\ast(h^\ast A, h^\ast M)
\]
describing the total (co)homology of the first pages in terms of the~$\gr_G\CO$-algebra structure on~$A$.
\end{theorem}
\begin{proof}
Using the coalgebraic version of Corollary~\ref{cor_negfilt} the filtration~\eqref{eq_coposfilt} gives a filtration of $\CO^\antis$ viewed as a comodule over itself.
The associated graded~$\CO^\antis$-comodule is isomorphic to $k^\ast \gr_G\CO^\antis$.
The $p$th graded piece, $(\gr_G\CO^\antis)_p$ is isomorphic to the subspace of $\CQ^\antis \circ \CP^\antis$ with weight $p$ contributed by $\CP^\antis$.
Hence $(\gr_G\CO^\antis)_p$ is concentrated in degrees $p$ and above.
Then part~(III) of Theorem~\ref{thm_cohomologyhodge} applies with the identity morphism $\CO\rightarrow\CO$ to give the desired spectral sequences.

For the second part we apply part~(I) of Theorem~\ref{thm_cohomologyhodge} to the morphism $h$ which for the homology gives
\[
 H^{\gr_G\CO}_\ast(h^\ast A, h^\ast N) \cong H_\ast(N \otimes_{\UO A} \LLO^{k^\ast\gr_G\CO^\antis}A),
\]
which is isomorphic to the total (co)homology of $E^1$.  The cohomological version is similar.
\end{proof}
\begin{example}
Let $A$ be an associative algebra, then with the filtered distributive law~$(\Com,\Lie,s,d)$ Theorem~\ref{thm_negfiltss} gives a spectral sequence converging to~$H^\As_\ast(A, M)$ for any bimodule $M$.
To calculate the first page we first note that $\gr_G\As$ is the operad for (non-unital) Poisson algebras and that the associative algebra $A$ is a Poisson algebra, $k^\ast\!A$ with zero commutative product and Lie bracket provided by $[a,b]= ab - ba$.
We may give the Poisson homology of this algebra explicitly in terms of a Chevalley-Eilenberg complex,
\begin{align*}
H^{\text{Pois}}_\ast(h^\ast A, h^\ast M) 
&\cong H^{\text{CE}}_\ast (\Sigma\coLie\circ A, M) \\
&\cong H^{\text{CE}}_\ast(A, M\otimes \Sigma\coCom\circ\Sigma\overline{\coLie}\circ A),
\end{align*}
where in the first Chevalley-Eilenberg complex the Lie algebra structure on~$\Sigma\coLie\circ A \cong A \oplus \Sigma\overline{\coLie}\circ A$ is given by a central extension of the Lie algebra structure on~$A$ with the adjoint action of~$A$ on~$\Sigma\overline{\coLie}\circ A$.
The second Chevalley-Eilenberg complex is obtained from an identity true for any central extension of a Lie algebra.
\end{example}

\subsection{Application to post-Lie algebras}
\label{section_postlie}
The operad $\PostLie$ was introduced by Bruno Vallette in~\cite{Vallette2007} as the Koszul dual of another operad~$\ComTrias$.
The algebras for $\PostLie$ have an anti-symmetric product $[-,-]$ which satisfies the Jacobi identity along with another binary operation $-\cc -$ without symmetries which satisfies the identities
\[
a\cc [b,c] = (a\cc b)\cc c - a\cc(b\cc c) - (a\cc c)\cc b + a\cc(c\cc b)
\]
and 
\[
[a, b]\cc c = [a,b\cc c] + [a\cc c, b].
\]
Thus $\PostLie$ is a quadratic operad.
An example of a post-Lie algebra is given by taking a Lie algebra, $\lieg$ and defining $g\cc h = 0$ for all $g,h\in\lieg$.
These examples are induced by the map $f:\PostLie\rightarrow\Lie$ given by sending the product $\cc$ to~$0$.  
Thus we will denote such examples by $f^\ast \lieg$.
However richer examples are provided in~\cite{MuntheKaas2013} where it is shown that vector fields equipped with a flat connection with constant torsion form a post-Lie algebra.
The torsion provides the Lie bracket, while the connection provides the other product, $x\cc y = \nabla_y x$.

It was shown in~\cite{Dotsenko2011fdls} that $\PostLie$ possesses a filtered distributive law $(\Lie,\Mag,s,d)$, where $\Mag$ is the operad for magmatic algebras, which have a binary product with no symmetries and which satisfies no identities.
The maps $s$ and $d$ are defined like so
\[
s([a,[b,c]] + [b,[c,a]] + [c,[a,b]]) = 0
\]
and
\begin{align*}
d([a,b]\cc c) &= [a\cc c, b] + [a,b\cc c] \\
d(a\cc [b,c]) &= (a\cc b)\cc c - a\cc(b\cc c) - (a\cc c)\cc b + a\cc(c\cc b).
\end{align*}
The Koszul dual operad $\ComTrias$ has a symmetric product $\co$ along with a product $\ca$ with no symmetries.
We will define it by the filtered distributive law $(\Nil, \Com, s',d')$ where $\Nil$ is the operad of nilpotent algebras which is generated by a product with no symmetries such that any multiplication involving three elements is equal to~0.
This is the Koszul dual operad of $\Mag$.
The map $s'$ is defined by
\[
s'(a\ca (b\ca c)) = s'((a\ca b)\ca c) = s'(a\ca(c\ca b)) = s'((a\ca c)\ca b) = a\ca(b\co c),
\]
while $d'$ is defined by
\begin{equation}
\label{eq_dComTrias}
d'(a\co (b\ca c)) = (a\co b)\ca c.
\end{equation}
The image of $d'$ is contained in $\Nil\circ\Com$ and hence the filtered distributive law for $\ComTrias$ is semi-filtered and the law for $\PostLie$ is op-semi-filtered.
We may therefore apply part (II) of Theorem~\ref{thm_cohomologyhodge} to the map $f:\PostLie\rightarrow\Lie$ to find that
for any Lie algebra $\lieg$, the cohomology of $\lieg$ viewed as a $\PostLie$-algebra decomposes:
let $M$ be a module for $\lieg$ then
\[
H_{\PostLie}^\ast(f^\ast\lieg, f^\ast M) \cong H^\ast_\Lie(\lieg,M) \oplus H^\ast(\Hom_{\UU_{\Lie}\lieg}(\LL_\Lie^{\Mag^\antis(2)\circ\Lie^\antis}\lieg, M)).
\]
There are only two terms because $\Mag^\antis \cong \susp\Nil \cong \CI\oplus sk\sym_2$ as an $\sym$-module.
Thus $\PostLie^\antis$ decomposes as $\Lie^\antis \oplus sk\sym_2 \circ \Lie^\antis$.
The first term is $\Lie^\antis$ and thus contributes the copy of the cohomology of $\lieg$ as a Lie algebra.
Examining the second factor further, the cotangent complex of this comodule has the form
\begin{equation}
\label{eq_Nilmodule}
\UU_\Lie\lieg \otimes_k \bigl( sk\sym_2 \circ \Lie^\antis \circ_\kappa \lieg \bigr),
\end{equation}
with an added differential which we will come to later.  Since $\Lie^\antis \cong \susp\Com$ as an $\sym$-module then $\Lie^\antis\circ_\kappa\lieg$ is the familiar Chevalley-Eilenberg complex $s^{-1}\bigwedge s\lieg$ used to define the Lie algebra (co)homology.
The effect of $sk\sym_2\circ(-)$ is to give two copies of this complex, hence as a graded module~\eqref{eq_Nilmodule} is isomorphic to 
\[
s^{-1}\UU_\Lie\lieg \otimes_k \bigwedge s\lieg \otimes_k \bigwedge s\lieg.
\]
We must now define the additional differential.  This makes use of the left comodule structure of $\PostLie^\antis$ given by taking the linear dual of~\eqref{eq_dComTrias}.
Since the image of $d'$ only involves the symmetric product of the left hand element the added differential on a general element of \eqref{eq_Nilmodule} is given by
\begin{multline}
d^l( 1 \otimes_k (g_1\wedge \ldots \wedge g_k)\otimes_k (h_1\wedge\ldots\wedge h_l)) = \\
\sum_{i=1}^k (-1)^{i+\beta_i-1} g_i\otimes (g_1\wedge \ldots \widehat{g_i} \ldots\wedge g_k) \otimes (h_1\wedge\ldots\wedge h_l),
\end{multline}
where $\beta_i$ is given by the sum of the degrees of the elements $g_1,\ldots,g_{i-1}\in\lieg$ and the notation $\widehat{g_i}$ means that the element $g_i$ is removed.
So along with the differential contributed by the two copies of $\Lie^\antis \circ_\kappa\lieg$ we find that $\LL_\Lie^{\Mag^\antis(2)\circ\Lie^\antis}\lieg$ is isomorphic as a left $\UU_\Lie\lieg$-module to $\LL_\Lie\lieg \otimes_k sB_\Lie(\lieg)$.

After a small amount of rearranging,
\[
H^\ast_\PostLie ( f^\ast \lieg, f^\ast M )
\cong H_\Lie^\ast (\lieg, M) \oplus \Hom_k(sB_\Lie(\lieg), H_\Lie^\ast(\lieg,M)).
\]
There is a corresponding result for homology with coefficients in a right $\UU_\Lie\lieg$-module $N$,
\[
H_\ast^\PostLie ( f^\ast \lieg, f^\ast N )
\cong H^\Lie_\ast (\lieg, N) \oplus sB_\Lie(\lieg)\otimes_k H^\Lie_\ast(\lieg,N).
\]
For a free Lie algebra \(\lieg = \Lie \circ V\) generated by a vector space~\(V\), this homology evaluates as \(V\otimes N \oplus sV\otimes V\otimes N\) which agrees with \(N\otimes \Mag^\antis \circ V\) given by Theorem~\ref{thmFree}.

\section{Equivalence with the Hodge-like decomposition of the Hochschild (co)homology of a commutative algebra}
\label{section_equivalence}
The classical Hodge-like decomposition of Hochschild homology or cohomology was described in terms of Eulerian idempotents~\cite{Gerstenhaber1987}, which were first used in this context by Barr~\cite{Barr1968} to study the Harrison (co)homology.
The Eulerian idempotents are elements of the symmetric group algebra; these act on the Hochschild (co)complex when the algebra is commutative; the classical approach is to show that they commute with the (co)differential and hence decompose the Hochschild (co)homology into the eigenspaces of the Eulerian idempotents.

In order to show that the classical decomposition can be recovered using our techniques we will show that in fact the Eulerian idempotents act on the $\sym$-module~$\As$ underlying the associative operad and commute with the $\Lie$-bimodule structure.
As such the Eulerian idempotents can be used to decompose the associative operad into a direct sum of $\Lie$-bimodules.  
We will show that this is the same direct sum as acheived by the theory of filtered distributive laws, see Example~\ref{ex_aslie}.

\newcommand{\sh}{{sh}}
The classical approach is as follows.
Take a subset $A\subseteq\sbrac{n}$ and write the elements in order as
\[
A = \cbrac{i_1 < \ldots < i_p}
\]
and the elements of the complement as
\[
\sbrac{n} \setminus A = \cbrac{j_1 < \ldots < j_q}.
\]
Define an element $\sigma_A\in\sym_n$ of the symmetric group on $n$ letters by
\[
\sigma_A(k) = \begin{cases}
i_k & \text{ if $k \leq p$,}\\
j_{k-p} & \text{ if $k > p$.}
\end{cases}
\]
Such an element is called a \emph{$(p,q)$-shuffle}.  Let $\sh_n\in k\sym_n$ be the element
\[
\sh_n = \sum_{A\subseteq\sbrac{n}} \sigma_A.
\]
For example $\sh_3$ is the sum of $2^3=8$ elements
\[
4[123] + [213] + [231] + [312] + [132].
\]
There is a left action of $k\sym_n$ on the $n$th group of the Hochschild cochain complex
\[
CC^n(A,M) \;\cong\; \Hom_k(A^{\otimes n}, M),
\]
which calculates the Hochschild cohomology of the associative algebra $A$ with coefficients in the $A$-bimodule $M$.
Barr~\cite{Barr1968} proved that when $A$ is commutative and $M$ is a symmetric bimodule the element $\sh_n$ is compatible with the Hochschild codifferential,
\[
\delta\sh_n = \sh_{n+1}\delta.
\]
Thus $\oplus_{n\geq 0} \sh_n$ acts on the Hochschild cochain complex.
In~\cite{Gerstenhaber1987} it is shown that each $\sh_n$ is diagonalisable and the eigenvalues of $\sh_n$ are determined to be $2^i$ for $i=1,\ldots,n$.
Writing $CC^n_{(i)}(A,M)$ for the $2^i$-eigenspace we get a decomposition
\[
CC^n(A,M) \cong CC^n_{(1)}(A,M) \oplus \ldots \oplus CC^n_{(n)}(A,M),
\]
or on the whole complex
\[
CC^\ast(A,M) \cong \bigoplus_{i\geq 1} CC^\ast_{(i)}(A,M).
\]
This is the Hodge-type decomposition studied in~\cite{Gerstenhaber1987}, there is also an analogue for the Hochschild homology, as well as cyclic homology.  There are other approaches to the decomposition, for example using Hopf algebras, see the approach of~\cite{Loday1998} where this is explained under the term $\lambda$-decomposition.

\subsection{Interpretation in terms of $\Lie$-bimodules}
We now show that the action of $\sh_n$ can be used to decompose the associative operad $\As$ in the category of $\Lie$-bimodules.
For any $\sym$-module $\CA$, there is a right $k\sym_n$-action on $\CA(n)$.  However for the associative operad $\As(n)\cong k\sym_n$ there is also a compatible left $k\sym_n$-action.
Thus the element $\sh_n\in k\sym_n$ acts on the $\sym$-module $\As(n)$ by left multiplication and in particular $\sh:=\oplus_{i\geq 1} \sh_i$ is a morphism of $\sym$-modules $\As\rightarrow\As$.

Via the inclusion $\Lie\rightarrow\As$ the associative operad becomes a $\Lie$-bimodule.
\begin{proposition}
\label{prop_shAs}
The morphism of $\sym$-modules $\sh:\As\rightarrow\As$ is also a morphism of $\Lie$-bimodules.
\end{proposition}
This will be proved below, but first we discuss its implications.
In characteristic~0, the eigenvalues of $\sh_m\in k\sym_m$ are all of the form $2^n$, so $\sh$ decomposes $\As$ into $2^n$-eigenspaces, each of which is a $\Lie$-bimodule.
To determine the eigenspaces, let $s_n = \frac{1}{n!}\sum_{\sigma\in\sym_n}\sigma$ be the averaged sum of all the elements of $\sym_n$.
This spans the trivial representation in $\As(n)$ and it is easily seen that $\sh_n s_n = 2^n s_n$.  
Thus $s_n$ is an element of the $2^n$-eigenspace for each $n$, as is any composition of $s_n$ with an element of $\Lie$.
Since $\As$ satisifies a filtered distributive law $(\Com, \Lie, s, d)$ with a splitting of $\sym$-modules $\Com\rightarrow\As$ picking out the elements $s_n$, we know that $\As$ is generated as a right $\Lie$-module by the elements $s_n$.
Hence the $2^n$-eigenspaces are generated by $s_n$ and furthermore the decomposition arising from the $\sh$ action agrees with the filtered distributive law decomposition.

\begin{proof}[Proof of Proposition~\ref{prop_shAs}]
As a $k\sym_n$-module $\As(n)$ is generated by $m_n$, the element describing the multiplication of $n$ ordered elements in an associative algebra.
We must show that for each $i$
\begin{equation}
\label{eq_Rlaction}
(\sh_n m_n)\circ_i l = \sh_{n+1}(m_n\circ_i l)
\end{equation}
and 
\[
l\cco (\sh_n m_n) = \sh_{n+1} (l\cco m_n),
\]
where $a\circ_i b$ is the element given by composing element $b$ in the $i$th position of element $a$ within the operad $\As$.
The second equation follows from the first since
\[
l\cco a = \sum_{j=1}^n a\circ_j l .(12\ldots j)
\]
for any element $a\in\As(n)$.  This is just the statement that within an associative algebra the adjoint action of some $a$ acts by derivations.

To show~\eqref{eq_Rlaction} we reduce the problem to calculations within the group algebra $k\sym_n$. The element $m_n$ corresponds to the identity $1\in k\sym_n$.
So the left hand side can be written as
\begin{equation}
(\sh_n m_n)\circ_i l= \sum_{A\subseteq \sbrac{n}} \sigma_A \circ_i l = \sum_{A\subseteq\sbrac{n}}(\sigma_A\circ_i m_2) (e - t_i)), 
\end{equation}
where $t_i$ is the transposition exchanging $i$ and $i+1$.
The right hand side of~\eqref{eq_Rlaction} is easily described since $m_n \circ_i l = (m_n\circ_i m_2)(e-t_i) = m_{n+1}(e-t_i)$, so the right hand side becomes $\sh_{n+1}(e-t_i)$.
To show the equality~\eqref{eq_Rlaction} we will show that $t_i$ acts trivially on the difference
\begin{equation}\label{eq_shuffledifference}
\sum_{A\subseteq\sbrac{n+1}} \sigma_A - \sum_{B\subseteq\sbrac{n}}\sigma_B\circ_i m_2.
\end{equation}
The action of $(-\circ_i m_2)$ on the shuffle $\sigma_B\in\sym_n$ serves to duplicate either a strand $i_p$ when $i\in B$ and $\sigma_B(p) = i$, or a strand $j_q$ when $i\notin B$ and $\sigma_B(q+\mbrac{B})=i$.
In either case $\sigma_B\circ_i m_2=\sigma_{B'}$ is still a shuffle with $B'=\brac{i_1<\ldots < i_p<i_p+1<\ldots< i_{\mbrac{B}}+1}$ in the first case and $B'=\brac{i_1<\ldots < i_k<i_{k+1}+1<\ldots< i_{\mbrac{B}}+1}$ in the second case where $k$ is such that $i_k < j_q < i_{k+1}$.

The set of shuffles obtained in this way are those $\sigma_A$ such that $\sigma_A^{-1}(i+1) = \sigma_A^{-1}(i)+1$, or equivalently the set of shuffles $\sigma_A$ such that either both $\sigma_A^{-1}(i),\sigma_A^{-1}(i+1)\in A$ or both $\sigma_A^{-1}(i),\sigma_A^{-1}(i+1)\notin A$.
Furthermore each of these shuffles can be obtained uniquely in the form $\sigma_B\circ_i m_2$.
Thus~\eqref{eq_shuffledifference} is equal to a sum of shuffles
\begin{equation}
\sum_{\substack{\sigma_A^{-1}(i)\in A\subseteq\sbrac{n+1} \\ 
\sigma_A^{-1}(i+1)\notin A}}\sigma_A
\quad + \sum_{\substack{\sigma_A^{-1}(i+1)\in A\subseteq\sbrac{n+1} \\
\sigma_A^{-1}(i)\notin A}}\sigma_A
\end{equation}
The action of $t_i$ exchanges the two summations and so acts trivially on~\eqref{eq_shuffledifference}.
\end{proof}

We have seen that the shuffle operator $\sh$ can be used to split the associative operad up into a direct sum of $\Lie$-bimodules.
Taking linear duals, an operator $\sh^\ast$ splits the coassociative cooperad into a direct sum of $\coLie$-comodules.
Applying Part~(II) of Theorem~\ref{thm_cohomologyhodge} gives the decompososition of the cohomology 
\[
H_\As^\ast(A,M) \cong H^\ast(\Hom_{\UU_\Com A}(\LL_{\Com}^{\susp\coAs}A, M))
\]
 for a commutative algebra~$A$ with symmetric bimodule~$M$.  
The image of the action of $\susp\sh^\ast:\susp\coAs\rightarrow\susp\coAs$ on the cohomology agrees with the classical action described by Barr~\cite{Barr1968}.
This shows that the classical decomposition and the decomposition achieved using the theory of operadic comodules agree.

\subsection{Extension to zinbiel algebras}
\label{section_dendzinb}
We now apply our results to the following square written alongside its Koszul dual
\[
\xymatrix{\Dend \ar[r] & \Zinb & & \Sigma\coDias \ar[r] & \Sigma\coLeib \\
 \As \ar[u]\ar[r] & \Com\ar[u] && \Sigma\coAs\ar[u] \ar[r] & \Sigma\coLie.\ar[u]}
\]
A decomposition for the dendriform (co)homology of a zinbiel algebra was described in~\cite{Goichot}.
We will recover this using the results of~\cite{Chapoton2001endo}.

The operad map $\Dend\rightarrow\Zinb$ is the Manin black product of $\As\rightarrow \Com$ with the pre-Lie operad, $\PLie$.
In the Koszul dual picture we have that $\coDias\rightarrow\coLeib$ is obtained from $\coLie\rightarrow\coAs$ by applying the Manin white product with $\coPerm$, the Koszul dual of $\PLie$, as expected.
But the white product of an operad with $\Perm$ is isomorphic to the Hadamard product with $\Perm$ and the same is true in the cooperad picture.

We now take advantage of the fact that if $\CM$ is a $\CC$-comodule for a cooperad $\CC$ then $\CM \otimes \coPerm$ is a $(\CC\!\otimes\!\coPerm)$-comodule.  
So writing $\Sigma\coAs$ in the category of $\coLie$-comodules as
\[
\coAs \cong \CM_1 \oplus \CM_2 \oplus \ldots,
\]
where $\CM_i\cong \coCom(i)\circ \coLie$,
we have a decomposition
\[
\coDias \cong \coAs \otimes \coPerm \cong \CM_1\otimes\coPerm \oplus \CM_2\otimes\coPerm \cong \ldots
\]
which translates to the decomposition of the (co)homology of a zinbiel algebra viewed as a dendriform algebra as described in~\cite{Goichot}.
In more generality if an operad $\CO$ is equal to the Manin black product of an operad $\CP$ and the operad $\PLie$ we may apply the same ideas to transfer results from $\CP$ to $\CO$.
\bibliographystyle{plain}
\bibliography{library}

\begin{thebibliography}{10}

\bibitem{Barr1968}
M.~Barr.
\newblock {Harrison homology, Hochschild homology and triples}.
\newblock {\em J. Algebra}, 8:314--323, 1968.

\bibitem{Chapoton2001endo}
F.~Chapoton.
\newblock Un endofoncteur de la cat\'egorie des op\'erades.
\newblock In {\em Dialgebras and related operads}, volume 1763 of {\em Lecture
  Notes in Math.}, pages 105--110. Springer, Berlin, 2001.

\bibitem{Chapoton2010}
F.~Chapoton.
\newblock Free pre-{L}ie algebras are free as {L}ie algebras.
\newblock {\em Canad. Math. Bull.}, 53:425--437, 2010.

\bibitem{Chapoton2001rooted}
F.~Chapoton and M.~Livernet.
\newblock {Pre-Lie algebras and the rooted trees operad}.
\newblock {\em International Mathematics Research Notices}, 2001(8):395--408,
  2001.

\bibitem{Dotsenko2011freeness}
V.~Dotsenko.
\newblock Freeness theorems for operads via {G}r\"obner bases.
\newblock {\em S\'eminaires et Congr\`es}, 26:61--76, 2011.

\bibitem{Dotsenko2011fdls}
V.~Dotsenko and J.~Griffin.
\newblock {Cacti and filtered distributive laws}.
\newblock {\em arXiv preprint math/1109.5345}, 2011.

\bibitem{Dotsenko2010grobner}
V.~Dotsenko and A.~Khoroshkin.
\newblock Gr\"obner bases for operads.
\newblock {\em Duke Math. J.}, 153(2):363--396, 2010.

\bibitem{Eilenberg1953}
S.~Eilenberg and S.~MacLane.
\newblock {On the groups H($\pi$, n), I}.
\newblock {\em Annals of Mathematics}, 58(1):55--106, 1953.

\bibitem{Fox1997}
T.~Fox and M.~Markl.
\newblock Distributive laws, bialgebras, and cohomology.
\newblock In {\em Operads: {P}roceedings of {R}enaissance {C}onferences
  ({H}artford, {CT}/{L}uminy, 1995)}, volume 202 of {\em Contemp. Math.}, pages
  167--205. Amer. Math. Soc., Providence, RI, 1997.

\bibitem{Fresse2009}
B.~Fresse.
\newblock {\em Modules Over Operads and Functors}.
\newblock Number v. 1967 in Lecture Notes in Mathematics. Springer, 2009.

\bibitem{Gerstenhaber1987}
M.~Gerstenhaber and S.~Schack.
\newblock {A Hodge-type decomposition for commutative algebra cohomology}.
\newblock {\em J. Pure Appl. Alg.}, 48(1--2):229--247, 1987.

\bibitem{Getzler1994operads}
E.~Getzler and J.D.S. Jones.
\newblock Operads, homotopy algebra and iterated integrals for double loop
  spaces.
\newblock {\em arXiv preprint hep-th/9403055}, 1994.

\bibitem{Getzler1995cyclic}
E.~Getzler and M.~Kapranov.
\newblock Cyclic operads and cyclic homology.
\newblock In {\em Geometry, topology, \& physics}, Conf. Proc. Lecture Notes
  Geom. Topology, IV, pages 167--201. Int. Press, Cambridge, MA, 1995.

\bibitem{Ginzburg1994}
V.~Ginzburg and M.~Kapranov.
\newblock Koszul duality for operads.
\newblock {\em Duke Math. J.}, 76(1):203--272, 1994.

\bibitem{Goichot}
F.~Goichot.
\newblock Decomposition of dendriform homology for zinbiel algebras.
\newblock 2001.

\bibitem{Hamilton2004}
A.~Hamilton and A.~Lazarev.
\newblock {Homotopy algebras and Noncommutative Geometry}.
\newblock {\em arXiv preprint math/0410621}, 2004.

\bibitem{Koszul1950}
J.-L. Koszul.
\newblock {Homologie et cohomologie des alg{\`e}bres de Lie}.
\newblock {\em Bull. Soc. Math. France}, 78:65--127, 1950.

\bibitem{Loday1989}
J.-L. Loday.
\newblock Op\'erations sur l'homologie cyclique des alg\`ebres commutatives.
\newblock {\em Invent. Math.}, 96:205--230, 1989.

\bibitem{Loday1998}
J.-L. Loday.
\newblock {\em {Cyclic homology}}, volume 301 of {\em Grundlehren der
  Mathematishchen Wissenschaften}.
\newblock Springer-Verlag, second edition, 1998.

\bibitem{Loday1984cyclic}
J.-L. Loday and D.~Quillen.
\newblock Cyclic homology and the {L}ie algebra homology of matrices.
\newblock {\em Commentarii Mathematici Helvetici}, 59(1):565--591, 1984.

\bibitem{Loday2012}
J.-L. Loday and B.~Vallette.
\newblock {\em {Algebraic Operads}}, volume 346 of {\em Grundlehren der
  mathematischen Wissenschaften}.
\newblock Springer-Verlag, 2012.

\bibitem{MuntheKaas2013}
H.~Munthe-Kaas and A.~Lundervold.
\newblock {On post-Lie algebras, Lie-Butcher series and moving frames}.
\newblock {\em Found. of Comp. Math.}, 13(4):583--613, 2013.

\bibitem{Natsume1989}
T.~Natsume and S.~Schack.
\newblock A decomposition for the cyclic cohomology of a commutative algebra.
\newblock {\em J. Pure Appl. Alg.}, 61:273--282, 1989.

\bibitem{Quillen1970}
D.~Quillen.
\newblock {On the (co)-homology of commutative rings}.
\newblock In {\em Applications of Categorical Algebra}, volume~17 of {\em Proc.
  Symp. Pure Math.}, pages 65--87. American Mathematical Society, 1970.

\bibitem{Vallette2007}
B.~Vallette.
\newblock Homology of generalized partition posets.
\newblock {\em J. Pure Appl. Algebra}, 208(2):699--725, 2007.

\end{thebibliography}

\end{document}